\documentclass{amsart}
\usepackage{amsmath,amsthm}
\usepackage{amsfonts,amssymb}
\usepackage{accents}
\usepackage{enumerate}
\usepackage{accents,color}
\usepackage{graphicx}
\usepackage{todonotes}
\usepackage{bm}

\newcommand{\lvt}{\left|\kern-1.35pt\left|\kern-1.3pt\left|}
\newcommand{\rvt}{\right|\kern-1.3pt\right|\kern-1.35pt\right|}

\newtheorem{thm}{Theorem}[section]
\newtheorem{cor}[thm]{Corollary}

\newtheorem{prop}[thm]{Proposition}

\theoremstyle{remark}
\newtheorem{rem}{Remark}[section]

 \def\la{{\langle}}
 \def\ra{{\rangle}}

 \def\d{\mathrm{d}}

 \def\a{{\alpha}}
 \def\b{{\beta}}
 \def\g{{\gamma}}

 \def\l{{\lambda}}
 
 \def\s{\sigma}
 \def\la{{\langle}}
 \def\ra{{\rangle}}

 \def\CB{{\mathcal B}}

 \def\CL{{\mathcal L}}

 \def\CV{{\mathcal V}}

 \def\NN{{\mathbb N}}

 \def\RR{{\mathbb R}}

 \def\YY{{\mathbb Y}}

\def\lla{\langle{\kern-2.5pt}\langle}      
\def\rra{\rangle{\kern-2.5pt}\rangle}

\newcommand{\wh}{\widehat}

\graphicspath{{./}}
\begin{document}

\title{Orthogonal polynomials on planar cubic curves}

\author{Marco Fasondini}
\address{Department of Mathematics\\
Imperial College\\
 London \\
 United Kingdom  }\email{m.fasondini@imperial.ac.uk}

\author{Sheehan Olver}
\address{Department of Mathematics\\
Imperial College\\
 London \\
 United Kingdom  }\email{s.olver@imperial.ac.uk}

\author{Yuan Xu}
\address{Department of Mathematics\\ University of Oregon\\
    Eugene, Oregon 97403-1222.}\email{yuan@uoregon.edu}

\date{\today}

\begin{abstract} 
Orthogonal polynomials in two variables on cubic curves are considered, including the case of elliptic curves.  
For an integral with respect to an appropriate weight function defined on a cubic curve, an explicit basis of 
orthogonal polynomials is constructed in terms of two families of orthogonal polynomials in one variable. 
We show that these orthogonal polynomials can  be used to approximate functions with cubic and square 
root singularities, and demonstrate their usage for solving differential equations with singular solutions.
\end{abstract}
\maketitle
 
\section{Introduction}
\setcounter{equation}{0}
 
We study orthogonal polynomials of two variables with respect to an inner product defined on a planar 
cubic curve. This is a continuation of recent work by the last two authors that studies orthogonal polynomials 
on simple one-dimensional geometries embedded in two-dimensional space, including wedges \cite{OX1} and 
quadratic curves \cite{OX2}, as well as higher-dimensional cases constructed via surfaces of revolution 
\cite{OX3,X20a,X20b}. 

It is assumed the cubic curve $\g$ is of the standard form $y^2 = \phi(x)$, where $\phi$ is a cubic 
polynomial of one variable, which includes the standard elliptic curves as a special case. We consider 
polynomials that are orthogonal with respect to the inner product 
$$
    \la f, g \ra_\g = \int_{\Omega_\g} f(x,y) g(x,y) w(x) \d \s(x,y),
$$ 
where $\Omega_\g$ is the set on which the cubic curve $\g$ is defined and $w$ is an appropriate weight 
function, and the inner product is well defined on the space $\RR[x,y]/\la y^2 - \phi(x)\ra$. These orthogonal 
polynomials are algebraic polynomials of two variables, but their structure is determined by the  
characteristics of the curve. In particular, the dimension of the space of the orthogonal polynomials of degree
$n$ is 3 for all $n \ge 3$, so that it does not increase with $n$, as with orthogonal polynomials of two variables 
on either an algebraic surface or on a domain with non-empty interior \cite{DX}. 

Our main result shows that the orthogonal structure on the cubic curve can be understood, by making use 
of symmetry, through a mixture of two univariate orthogonal systems. To wit, we are able to construct 
orthogonal polynomials on the curve explicitly in terms of univariate orthogonal polynomials. Moreover, this 
structural connection propagates to quadrature rules and polynomial interpolation based on the roots of the orthogonal
polynomials. Hence, we have developed a toolbox for the computational and analytical 
study of functions on cubic curves. We provide two applications to showcase the usage of our results. 
The first is the approximation of functions with cubic singularities. We compare the results to Hermite--Pad\'e approximation, a common technique for approximating functions with singularities, demonstrating that our approximation converges faster and is more robust to degeneracies. The 
second application is differential equations, which demonstrates the effectiveness of a spectral 
collocation method, based on our toolbox, for solving differential equations with singular solutions. The examples include 
the computation of an elliptic integral that can be expressed in terms of the Legendre's incomplete integral of the first kind.

The paper is organized as follows. The orthogonal structure on the cubic curve is established in the next section,
where we clarify the families of cubic curves that we consider, which leads to several distinguished cases, 
and show how orthogonal polynomials can be constructed explicitly in terms of univariate orthogonal polynomials; 
the section also contains several families of examples. In the third section we consider quadrature rules on 
the cubic curve as well as polynomial interpolation based on the nodes of the quadrature rules, both from a 
theoretical and a computational perspective. The applications are discussed in the fourth section and the final section is on possible future work.

{\bf Acknowledgment}. The first and second authors were supported by the Leverhulme Trust Research Project Grant RPG-2019-144 ``Constructive approximation theory on and inside algebraic curves and surfaces''.

\section{Orthogonal polynomials on cubic curves}\label{sect:OPscubiccurves}
\setcounter{equation}{0}

\subsection{Cubic curves}\label{sect:cubiccurves}
Throughout this paper we let $\phi$ be a cubic polynomial defined by 
\begin{equation}\label{eq:cubic-curve}
  \phi(x) = a_0 x^3 + a_1 x^2 + a_2 x + a_3, \qquad a_i \in \RR, \quad a_0 \ne 0.
\end{equation}
We consider the cubic curve $\g$ on the plane defined by the standard form
$$
    y^2 = \phi(x), \qquad (x,y) \in \RR^2,  
$$
since all irreducible bivariate cubics can be transformed\footnote{While the transformation to canonical form will not necessarily map polynomials to polynomials, it will still provide an orthogonal expansion on other cubic curves.}  into this form~\cite{bix,BM}. We let $\g = \{(x,y): y^2 = \phi(x)\}$ be
the graph of the curve. Without loss of generality, we assume that $a_0>0$, so that $\phi(x) > 0$ for sufficiently 
large $x > 0$. Let 
$$
   \Omega_\g:= \{x: \phi(x) > 0\},
$$
which is the set on which the cubic curve is defined. The cubic polynomial $\phi$ can have either one 
real zero or three real zeros, so that $\Omega_\g$ can be either one interval or the union of two
intervals. This leads to three possibilities: 
\begin{enumerate}[    (I)]
\item $\Omega_\g = (A,\infty)$: the curve has one component; 
\item $\Omega_\g = (A_1, B_1) \cup (A_2, \infty)$ with $A_1 < B_1 < A_2$: the curve has two disjoint components; 
\item $\Omega_\g = (A,B) \cup (B,\infty)$: the curve has two touching components. 
\end{enumerate}
In the first case, $\phi$ has one real zero $A$. In the second case, $\phi$ has three real
zeros $A_1 < B_1 < A_2$. In the third case, $\phi$ has a real zero at $A$ and a double zero at $B$.
For examples, see Figures 1 and 2 below.

One important family of cubic curves included in our definition is that of {\it elliptic curves}. 
An elliptic curve is a plane curve defined by 
\begin{equation} \label{eq:elliptic1}
     y^2 = x^3 + a x + b, 
\end{equation}
where $a$ and $b$ are real numbers and the curve has no-cusps, self-intersections, or isolated points. This 
holds if and only if the discriminant
$$
\Delta_E = -16(4a^{3}+27b^{2})
$$
is not equal to zero. The graph has two components if $\Delta_E > 0$ and one component if $\Delta_E < 0$.
Two elliptic curves are depicted in Figure 1, the left-hand one has one component, whereas the right-hand one 
has two components. 
\begin{figure}[ht]
  \begin{center}
    \includegraphics[width=0.45\textwidth]{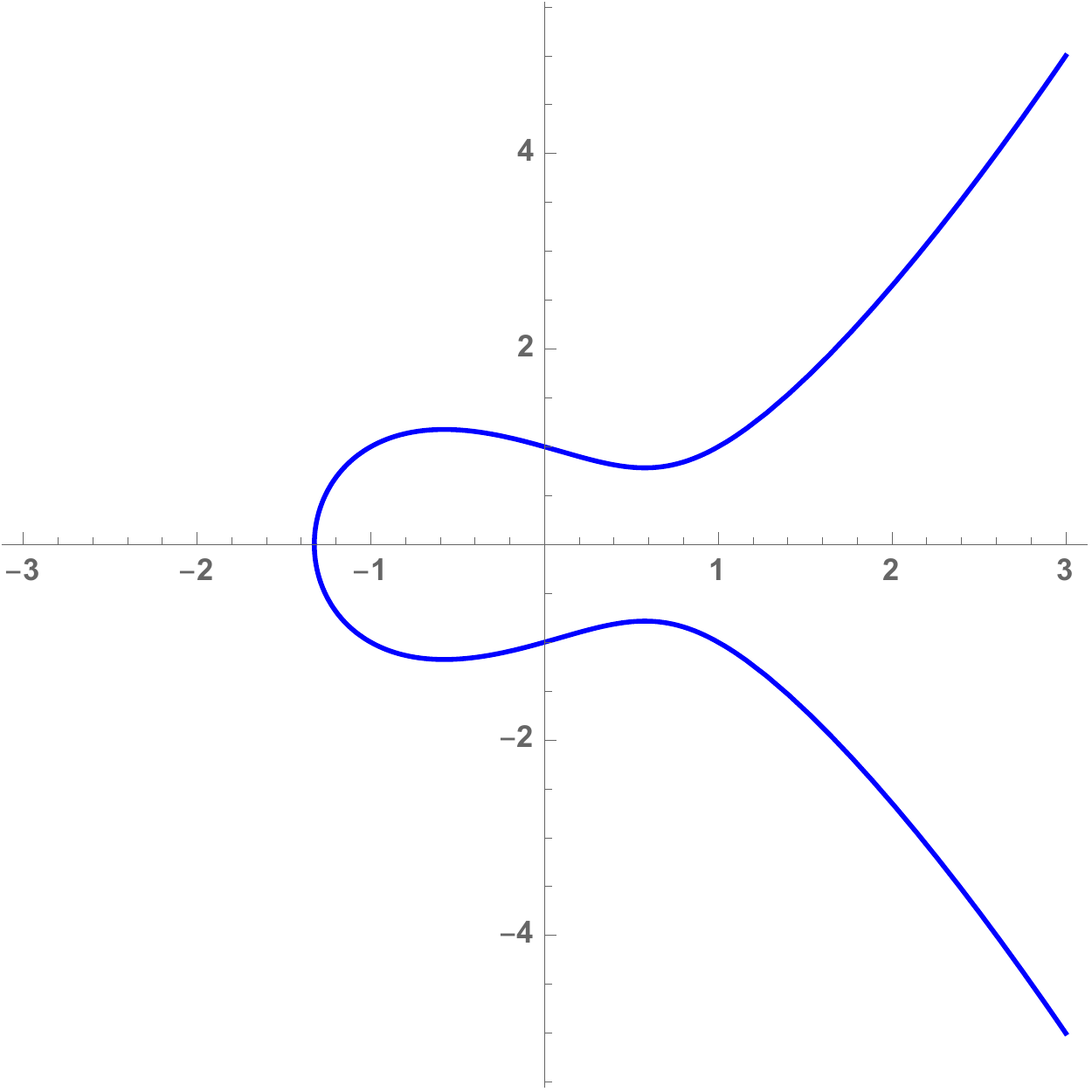} \qquad \includegraphics[width=0.45\textwidth]{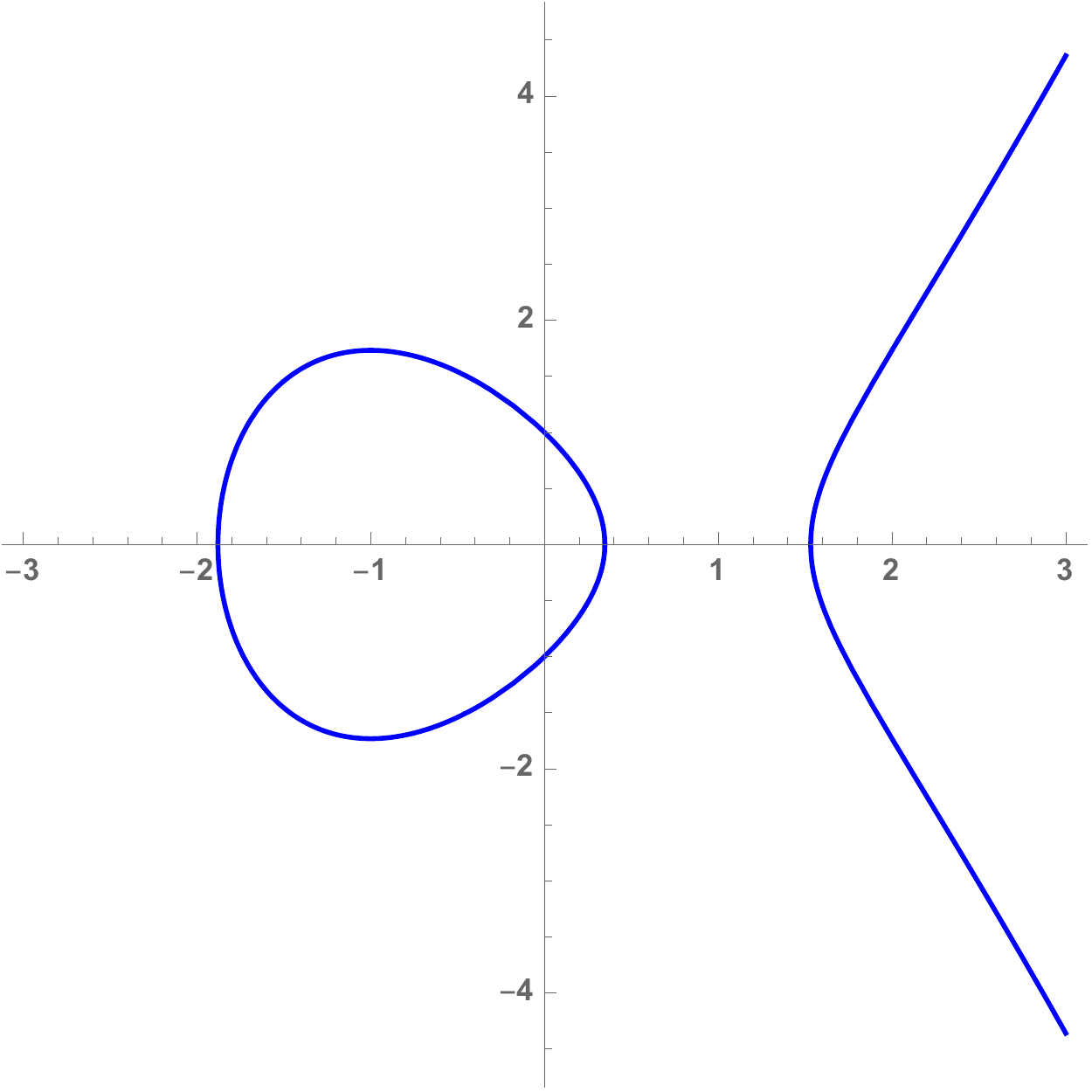}
  \end{center} 
  \caption{Elliptic curves. Left: $y^2=x^3-x+1$. Right: $y^2=x^3-3x+1$}
\end{figure}

We can write the elliptic curve in a different form. Let $\phi(x) = x^3 + a x + b$. By the definition, $\phi(A) = 0$ 
or $b = - A^3 - a A$, so that 
$$
  \phi(x) =  x^3 + a x - A^3 - a A = (x-A) (x^2 + A x +A^2+a). 
$$
We need $x^2 + A x + A^2 + a \ge 0$ for $x \ge A$, which holds only if its discriminant $A^2- 4 (A^2+a) < 0$ 
or $a > - (3/4)A^2$. Under this condition, we have 
$$
 4 a^3 + 27 b^2 = 4 a^3 + 27 A^2 (A^2+a)^2 > - 4 \frac{3^3}{4^3} A^6 +27 A^2 \left(\frac{A^2}{4}\right)^2 = 0,
$$ 
so that the elliptic curve does indeed have one component. Thus, setting $c = a + \frac{3}{4}A^2$ and 
$b = - A^3 - a A$ we see that the elliptic curve \eqref{eq:elliptic1} becomes 
\begin{equation} \label{eq:elliptic2}
      y^2 = (x-A) \left( \left (x+\tfrac{A}{2}\right)^2 + c\right), \qquad A\in \RR.
\end{equation}
If $c > 0$, then $\phi(x) = (x-A) (x+\frac{A}{2})^2 + c$ has one real zero, so that the elliptic curve has one
component. If $c < 0$, then the curve has two components. One of the advantages of writing the curve in 
the form \eqref{eq:elliptic2} is that the roots of $\phi$ are explicitly given. 

We also consider cubic curves that are not elliptic. For example, we can have cubic curves of the form 
$y^2 = a (x^3 - b x^2)$, which will self intersect when $b>0$ and will be treated as two components that 
touch at one point. One example of such curves is depicted in Figure 2. 

As an example of the third case, that of closed cubic curves, we mention tear drop curves defined by 
$$
     (2a)^2 y^2 = (x-a)^2 (x+a), \qquad -a < x < a. 
$$
For $a=1$, this curve is inside the unit circle and is depicted in Figure 2. 
\begin{figure}[ht]
  \begin{center}
    \includegraphics[width=0.38\textwidth]{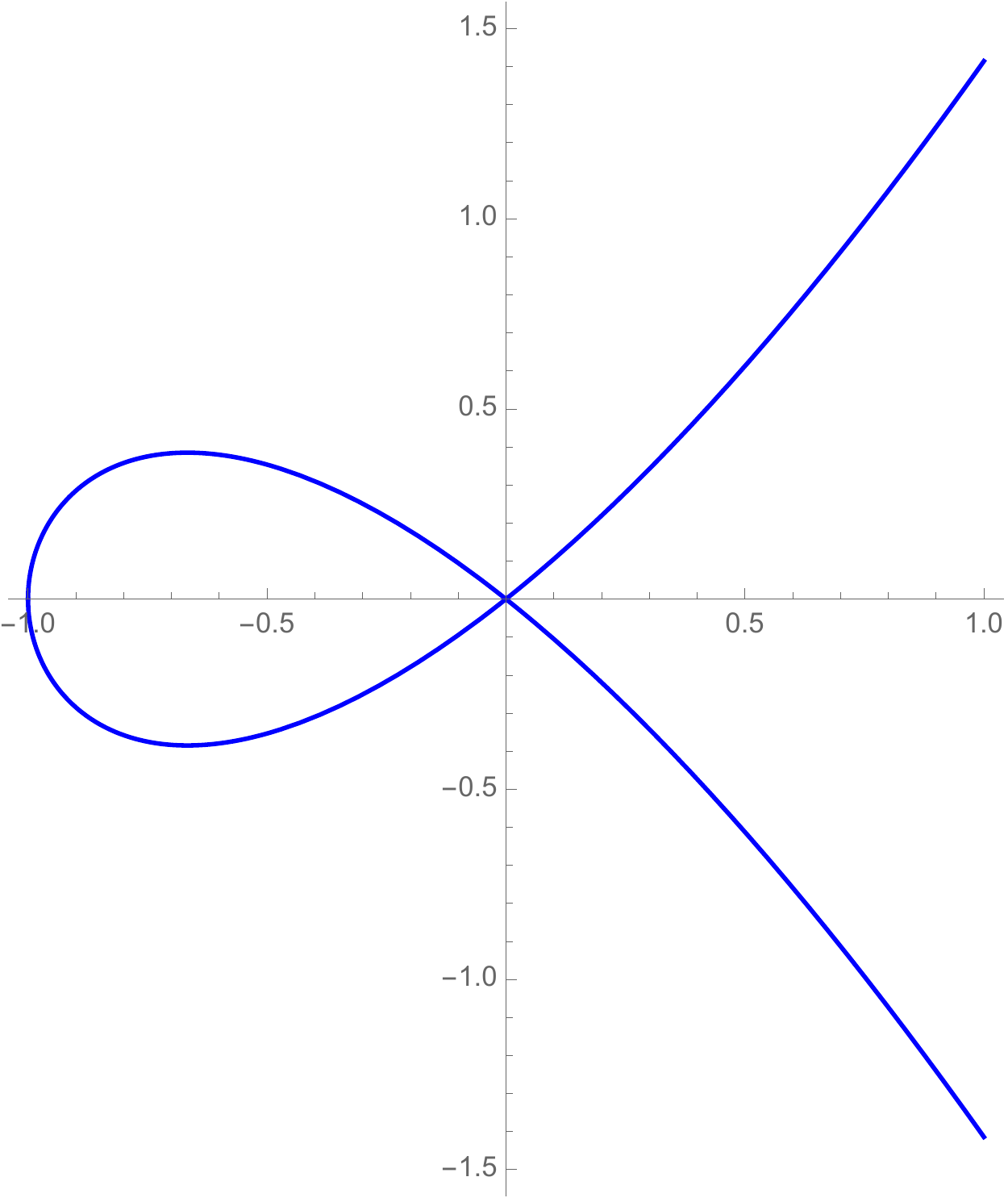} \qquad\quad  \includegraphics[width=0.45\textwidth]{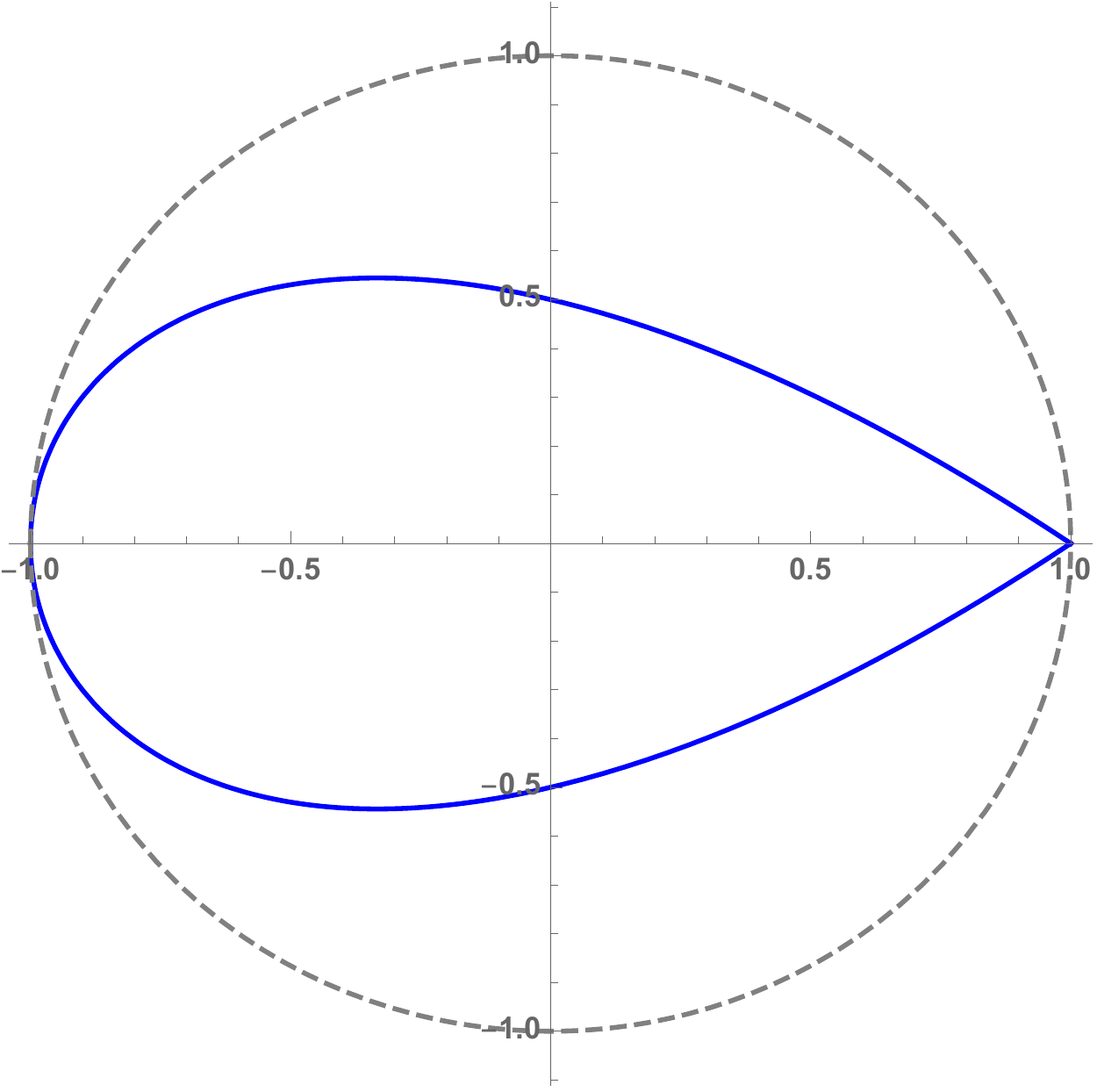}
  \end{center} 
  \caption{Left:  $y^2= x^2(x+1)$. Right:  the tear drop curve \mbox{$4 y^2= (1-x)^2 (1+x)$}}
\end{figure}

Our definition also include the curve $y^2 = x^3$, which is the case of $a= b = 0$ in \eqref{eq:elliptic1},
but the curve has a singular point and is not an elliptic curve. 

\subsection{Orthogonal polynomials on cubic curves}
Let $y^2=\phi(x)$ be a cubic curve and let $w$ be a non-negative weight function defined on $\Omega_\g$. 
We consider orthogonal polynomials of two variables that are orthogonal with respect to an inner product 
defined, on an appropriate polynomial subspace, by 
\begin{equation} \label{eq:ipd1}
  \la f,g\ra_{\g,w} = \int_{\g} f(x,y) g(x,y) w(x) \d \s(x,y),
\end{equation}
where $\d \s$ is the arc length measure on the curve. Depending on the support set of $w$, the integral
domain could be compact in the cases I and II. 

The bilinear form $\la \cdot,\cdot\ra_{\g,w}$ defines an inner product on the space $\RR[x,y] / \la y^2 - \phi(x)\ra$. 
For $n\ge 3$, the monomials of degree exactly $n$, $x^k y^{n-k}$ for $0 \le k \le n$, remain of degree $n$ 
 modulo the ring $\la y^2 - \phi(x)\ra$ only when $k =0,1,2$. In particular, this shows that 
$\CB_n = \{y^n, x y ^{n-1}, x^2 y^{n-2}\}$ is a basis of the space of polynomials of degree exactly $n$ in 
$\RR[x,y] / \la y^2 - \phi(x)\ra$. 

Let $\CV_n:=\CV_n(\g,w)$ be the space of orthogonal polynomials of degree $n$ in two variables with respect to this 
inner product. Applying the Gram--Schmidt process to the basis $\CB_n$, for example, inductively on $n$, 
we obtain the following proposition:

\begin{prop}\label{prop:dimVn}
For $n \in \NN_0$, we have $\dim \CV_0 =1$, $\dim \CV_1 =2$ and 
$$
   \dim \CV_n =3, \quad n \ge 2.
$$
\end{prop}

Let $p_n(w)$ be a family of univariate orthogonal polynomials with respect to the standard inner product on $\Omega_\g$,
\begin{align} \label{eq:ipduni}
  \la f,g \ra_{w} = \int_{\Omega_\g}    f\left(x \right) g\left(x \right) w(x) \d x. 
\end{align}
 In particular,
$p_n(\phi w)$ denotes orthogonal polynomials with respect to $\phi(x) w(x)$ on $\Omega_\g$. A basis for $\CV_n$ can be given explicitly in terms of orthogonal polynomials with respect to $w$ and $\phi w$. We parametrize the inner product (\ref{eq:ipd1}) as 
\begin{align} \label{eq:ipd2}
  \la f,g \ra_{\g,w} = \int_{\Omega_\g} & \left[  f\left(x, \sqrt{\phi(x)} \right) g\left(x, \sqrt{\phi(x)} \right) \right. \\
      & \left. + f\left(x, - \sqrt{\phi(x)} \right) g\left(x, -\sqrt{\phi(x)} \right)\right] w(x) \d x. \notag
\end{align}
More precisely, the domain $\Omega_\g$ in the integral should be replaced by $\mathrm{supp}(w) \subset \Omega_\g$,
where $\mathrm{supp}(w)$ denotes the support set of $w$. For example, in case $I$, we could choose $w$ 
so that it has support set $[A,B]$ for some $B \in \RR$.

We now define an 
explicit basis for the space $\CV_n(w)$ of orthogonal polynomials on the cubic curve $\g$. We denote this basis by 
$Y_{n,i}$ and denote its squared norm by $H_{n,i}$. The squared norm of $p_n(w;x)$ is denoted by $h_n(w) = \la p_n, p_n \ra_{w}$.

\begin{thm} \label{thm:OPbasis}
Let $\g$ be a cubic curve and let $w$ be a weight function defined on $\Omega_\g$.
\begin{enumerate}[1.]
\item For $n =0$ and $n=1$, we define
$$
Y_0(x,y) =1, \qquad Y_{1,1}(x,y) = p_1(w;x), \qquad Y_{1,2}(x,y)=y.
$$
Then $\CV_0 = \mathrm{span} \{Y_0\}$ and $\CV_1 =  \mathrm{span} \{Y_{1,1}, Y_{1,2}\}$. Moreover,
$$
   H_0 = 2 h_0(w), \qquad H_{1,1} = 2 h_1(w), \qquad H_{1,2} = 2 h_0(\phi w).
$$
\item For $m \in \NN_0$ and $m \ge 1$, we define 
\begin{align*}
  Y_{2m,1}(x,y) &\, = p_{3m}(w;x), \\
  Y_{2m,2}(x,y) &\,= p_{3m-1}(w;x), \\
  Y_{2m,3}(x,y) &\,= y p_{3m-2}(\phi w; x),
\end{align*}
and 
\begin{align*}
  Y_{2m+1,1}(x,y) &\, = p_{3m+1}(w;x), \\
  Y_{2m+1,2}(x,y) &\,= y p_{3m}(\phi w;x), \\
  Y_{2m+1,3}(x,y) &\,= y p_{3m-1}(\phi w; x).
\end{align*}
Then $Y_{n,i}$ is a polynomial of degree $n$ in $\RR[x,y] / \la y^2 - \phi(x)\ra$ for $i=1,2,3$ and
$$
\CV_n = \mathrm{span} \{Y_{n,1}, Y_{n,2}, Y_{n,3} \}, \qquad n \ge 2.
$$
Moreover,  the norms of these polynomials are given by
\begin{align*}  
H_{2m,1} & = 2 h_{3m}(w), \quad H_{2m,2} = 2 h_{3m-1}(w), \quad H_{2m,3} = 2 h_{3m-2}(\phi w),\\
H_{2m+1,1} & = 2 h_{3m+1}(w), \quad H_{2m+1,2} = 2 h_{3m}(\phi w), \quad H_{2m+1,3} = 2 h_{3m-1}(\phi w).
\end{align*}
\end{enumerate}
\end{thm}

\begin{proof}
For $n \ge 2$, we first show that $Y_{n,i}$ is of degree $n$ in $\RR[x,y] / \la y^2 - \phi(x)\ra$. Throughout 
this proof, we introduce a function $\psi(x,y)$ so that the equation of the cubic curve becomes
$$
     \psi (x,y) = x^3, \qquad \psi(x,y) = a_0^{-1}( y^2 - a_1 x^2-a_2 x - a_3).
$$
Now, for $n =2m$, we write
\begin{align*}
& p_{3m}(w;x) = \sum_{k=0}^{3m} b_k x^k  = b_0 + b_1 x + \sum_{j=1}^{m} b_{3j-1} x^{3j-1} +
  \sum_{j=1}^{m} b_{3j} x^{3j} +  \sum_{j=1}^{m-1} b_{3j+1} x^{3j+1} \\
  &\qquad = b_0 + b_1 x + \sum_{j=1}^{m} b_{3j-1} (\psi(x,y))^{j-1}x^2 +
  \sum_{j=1}^{m} b_{3j} (\psi(x,y))^j + \sum_{j=1}^{m-1} b_{3j+1} x (\psi(x,y))^j, 
\end{align*}
which is a polynomial of degree $2m$ in $x,y$ variables. The same argument also shows that $p_{3m-1}(w;x)$
is a polynomial of degree $2m$ in $\RR[x,y] / \la y^2 - \phi(x)\ra$. Furthermore, it follows that $p_{3m-2}(w;x)$
is a polynomial of degree $2m-1$ in $\RR[x,y] / \la y^2 - \phi(x)\ra$, which shows that $Y_{2m,3}$ is a 
polynomial of degree $2m$ mod $ \la y^2 - \phi(x)\ra$. A similar argument works for $Y_{2m+1,i}$.

We now verify orthogonality. The $Y_{n,i}$ are of two forms: either $Y_{n,i}(x,y) = f(x)$ or $Y_{n,i}(x,y) = yg(x)$.   due to symmetry in the expression of $\la \cdot,\cdot\ra_{\g,w}$
in \eqref{eq:ipd2},
\begin{equation} \label{eq:ipd-parity}
  \la f(x), y g(x) \ra_{\g,w} = 0. 
\end{equation} 
This establishes orthogonality between all $Y_{n,i}$ of the form $f(x)$ and those of the form $yg(x)$. To prove orthogonality between the remaining $Y_{n,i}$ (those that are both of the form $f(x)$ and those that are both of the form $yf(x)$), we relate the bivariate inner product on $\g$ to the univariate inner product on $\Omega_{\gamma}$ in both cases:  observe that
\begin{align}
\la f(x),  g(x) \ra_{\g,w} = 2\la f(x),  g(x) \ra_{w}, \label{eq:xip1}
\end{align}
and
\begin{align}
\la yf(x),  yg(x) \ra_{\g,w} = 2\la y^2f(x),  g(x) \ra_{w}= 2\la \phi f(x),  g(x) \ra_{w} = 2\la f(x),  g(x) \ra_{\phi w}.  \label{eq:xyip1}
\end{align}
Thus, the orthogonality between the remaining $Y_{n,i}$ follow from the orthogonality of the $p_n(w)$ and $p_n(\phi w)$. Hence we have showed that the $Y_{n,i}$ form a basis for $\CV_n$, $n \geq 0$. It follows from the last two equations that if  $Y_{n,i} = p_k(w)$, then  $H_{n,i} = \la Y_{n,i}, Y_{n,i} \ra_{\g,w}= 2\la p_k(w), p_k(w) \ra_{w} = 2h_{k}(w)$. If $Y_{n,i} = yp_k(\phi w)$, then $H_{n,i} = 2\la p_k(\phi w), p_k(\phi w) \ra_{\phi w} = 2h_{k}(\phi w)$.
\end{proof}

\subsection{Fourier orthogonal series}
For $w$ defined on $\RR$, the Fourier orthogonal series in terms of orthogonal polynomials $\{p_n(w)\}$ is 
defined by 
$$
   f = \sum_{n=0}^\infty \wh f_n(w) p_n(w), \qquad p_n(w) = \frac{1}{h_n(w)} \int_\RR f(t) p_n(w;t) w(t) \d t,
$$
where the identity holds in $L^2(w)$ as long as polynomials are dense in $L^2(w)$, which we assume to be
the case. Furthermore, let $s_n(w;f)$ denote the $n$-th orthogonal partial sum of this expansion; that is, 
$$
 s_n(w;f) = \sum_{k=0}^n \wh f_k(w) p_k(w), \qquad n =1,2,\ldots.
$$

Likewise, let $\g$ be a cubic curve and $w$ be a weight function defined on $\g$, we define the Fourier 
orthogonal series of $f\in L^2(\g, w)$ by
$$
  f = \wh f_0 Y_0 +  \wh f_{1,1}Y_{1,1} +  \wh f_{1,2}Y_{1,2} +
  \sum_{n=2}^\infty \sum_{i=1}^3 \wh f_{n,i} Y_{n,i}, \quad \hbox{with} \quad 
\wh  f_{n,i} =\frac{ \la f, Y_{n,i}\ra_{\g,w}}{H_{n,i}(w)},
$$
The $n$-th partial sum of this expansion is denoted by $S_n(w;f)$; that is 
$$
S_n(w; f) = \wh f_0 Y_0 +  \wh f_{1,2}Y_{1,1} +  \wh f_{1,2}Y_{1,2} +
  \sum_{k=2}^n \sum_{i=1}^3 \wh f_{k,i} Y_{k,i}.
$$
The next theorem shows that this partial sum can be written in terms of the partial sums of orthogonal 
series with respect to $w$ and $\phi w$. Let $\|\cdot\|_w$ denote the norm of $L^2(\g,w)$.
%
\begin{thm}\label{thm:partsums}
Let $\g$ be a cubic curve and let $w$ be a weight function defined on $\Omega_\g$. For $f \in 
L^2(\g,w)$, define 
\begin{align} \label{eq:fe-fo}
\begin{split}
  f_e(x):= & \frac{f\left(x, \sqrt{\phi(x)}\right)+ f\left(x, - \sqrt{\phi(x)}\right)}{2}, \\
  f_o(x):= & \frac{f\left(x, \sqrt{\phi(x)}\right)- f\left(x, - \sqrt{\phi(x)}\right)}{2 \sqrt{\phi(x)}}
\end{split}
\end{align}
for $x \in \Omega_\g$. Then
\begin{align*}
\begin{split}
  S_{2m} (w; f, x, y) &= s_{3m} (w; f_e, x) + y s_{3m-2} (\phi w; f_o, x), \qquad (x,y) \in \g, \\
  S_{2m+1} (w; f, x, y) & = s_{3m+1} (w; f_e, x) + y s_{3m} (\phi w; f_o, x), \qquad (x,y) \in \g.
\end{split}
\end{align*}
Furthermore, the $L^2(\g,w)$ norm of $S_n(w;f)$ satisfies
\begin{align}\label{eq:norm-partialsum}
\begin{split}
  \left \|S_{2m}(w; f)\right \|_{\g,w}^2 & = 2 \|s_{3m} (w; f_e) \|_{w}^2 +2 \| s_{3m-2} (\phi w; f_o)\|_{\phi w}^2, \\
  \|S_{2m+1}(w; f)\|_{\g,w}^2 & = 2 \|s_{3m+1} (w; f_e) \|_{w}^2 +2 \| s_{3m} (\phi w; f_o)\|_{\phi w}^2. 
\end{split}
\end{align}
\end{thm}

\begin{proof}
It follows from the definitions of $f_e$ and $\la \cdot, \cdot \ra_{\g,w}$ that
$$
 \la f, Y_{2m,1} \ra_{\g,w} =  2\la f_e, p_{3m}(w) \ra_{w},
$$
Since $H_{2m,1} = 2h_{3m}(w)$, we obtain $\wh f_{2m,1} = \{\wh f_e\}_{3m}(w)$. The same argument
shows also that $\wh f_{2m,2} = \{\wh f_e\}_{3m-1}(w)$, $\wh f_{2m+1,1} = \{\wh f_e\}_{3m+1}(w)$, $\wh f_0 = \{\wh f_e\}_0(w)$ and $\wh f_{1,1} = \{\wh f_e\}_1(w)$. Furthermore, from the definitions of $f_o$ and $\la \cdot, \cdot \ra_{\g,w}$  we have 
$$
 \la f, Y_{2m,3} \ra_{\g,w} =  2\la f_o, p_{3m-2}(\phi w) \ra_{\phi w}.
$$
so that $\wh f_{2m,3} = \{\wh f_o\}_{3m-2}(\phi w)$.  The same argument also shows that $\wh f_{2m+1,2} =
 \{\wh f_o\}_{3m}(\phi w)$, $\wh f_{2m+1,3} = \{\wh f_o\}_{3m-1}(\phi w)$ and $\wh f_{1,2} = \{\wh f_o\}_{0}(\phi w)$. 

Setting $F_k(x) =  \{\wh f_e\}_k(w) p_k(w;x)$ and $G_k(x) =  \{\wh f_o\}_k(\phi w) p_k(\phi w;x)$, we can 
write the partial sum with $n = 2m$ as 
\begin{align*}
 S_{2m}(w; f)(x,y)  = & F_0(x) + F_1(x) + y G_0(x)  + \sum_{k=1}^{m} \left[G_{3k+1}(x) 
   +F_{3k-1}(x) + y G_{3k-2} (x) \right]\\
    & + y \sum_{k=1}^{m-1} \left[F_{3k+1}(x) + y G_{3k}(x) + y G_{3k-1} (x) \right]\\
   = & \sum_{k=0}^{3m+1} F_k (x) + y \sum_{k=0}^{3m-2} G_k(x) =
   s_{3m} (w; f_e, x) + y s_{3m-2} (w; f_o, x).
\end{align*}
A similar proof works for $S_{2m+1}(w;f)$. By \eqref{eq:ipd-parity} and the Parseval identity, we see that 
\begin{align*}
 \|S_{2m}(w;f)\|_w^2 & = \|s_{3m} (w; f_e, x) \|_{\g,w}^2 + \| y s_{3m-2} (\phi w; f_o, x)\|_{\g,w}^2 \\
      & =  2\|s_{3m} (w; f_e) \|_{w}^2 + 2\| s_{3m-2} (\phi w; f_o)\|_{\phi w}^2,
\end{align*}
where the second identity follows since $|s_{3m} (w; f_e, x)|^2$ does not contain $y$, whereas
$|y s_{3m} (\phi w; f_o, x)|^2$ contains a $y^2$, which is equal to $\phi(x)$. The proof for the norm of 
$S_{2m+1}(w;f)$ is similar. 
This completes the proof. 
\end{proof}

\begin{cor}\label{cor:psumconv}
Let $\g$ be a cubic curve and let $w$ be a weight function defined on $\Omega_\g$. Let $f \in 
L^2(\g,w)$. Then $S_n(w; f)$ converges to $f$ in $L^2(\g,w)$. 
\end{cor}

\begin{proof}
For $f\in L^2(\g,w)$, it follows from \eqref{eq:ipd2} and $|a+b|^2 \le 2 (|a|^2+|b|^2)$ that 
$$
  \|f_e\|_w^2 \le \frac12 \int_{\Omega_f} \left[|f\big (x,\sqrt{\phi(x)} \big)|^2 + |f\big (x,- \sqrt{\phi(x)} \big)|^2\right]w(x)\d x
     = \frac12 \|f\|_{\g,w}^2
$$
and similarly $\|f_o\|_{\phi w}^2 \le  \frac12 \|f\|_{\g,w}^2$. Hence, $f_e \in L^2(w)$ and $f_o \in L^2(\phi w)$. 
It follows that $s_n(w; f_e)$ converges to $f_e$ in $L^2(w)$ and $s_n(w; f_o)$ converges to 
$f_o$ in $L^2(\phi w)$. Consequently, the convergence of $S_n(w; f)$ in $L^2(\g,w)$ follows from 
\eqref{eq:norm-partialsum}.
\end{proof}

\subsection{Jacobi operators}\label{sect:jacops}
Let $\wh Y_{n,i} = Y_{n,i}/\sqrt{H_{n,i}}$. Then $\{\wh Y_{n,i}\}$ is an orthonormal basis of $\CV_n$. 
Define 
$$
    \wh\YY_0 =\left[\wh Y_0\right], \quad \wh\YY_1 = \left[ \begin{matrix} \wh Y_{1,1} \\ \wh Y_{1,2}  \end{matrix} \right]
    \quad \hbox{and}\quad \wh\YY_n = \left[ \begin{matrix} \wh Y_{n,1} \\ \wh Y_{n,2} \\ \wh Y_{n,3} \end{matrix} \right], \quad n \ge 2.
$$ 
The general theorem of orthogonal polynomials of several variables shows that
\begin{align} 
  x \wh\YY_n &\, = A_{n,1} \wh\YY_{n+1} + B_{n,1} \wh\YY_n + A_{n-1,1}^t \wh\YY_{n-1},  \label{eq:xrec}\\
 y \wh\YY_n &\, = A_{n,2} \wh\YY_{n+1} + B_{n,2} \wh\YY_n + A_{n-1,2}^t \wh\YY_{n-1}, \label{eq:yrec}
\end{align}
where $A_{0,i}$ are $1\times 2$ matrices and $B_{n,0}$ is a real number; $A_{1,i}$ are $2\times 3$ matrices
and $B_{1,i}$ are $2\times 2$ matrices; $A_{n,i}$ and $B_{n,i}$ are $3\times 3$ matrices for all $n \ge 2$.
The matrices $A_{n,i}$ and $B_{n,i}$ are determined by orthogonality relations: 
$$
   A_{n,1} = \la x  \wh\YY_n \wh\YY_{n+1}^t \ra_{\g,w}, \: A_{n,2} = \la y  \wh\YY_n \wh\YY_{n+1}^t \ra_{\g,w}, \: 
    B_{n,1} = \la x  \wh\YY_n \wh\YY_{n}^t \ra_{\g,w}, \: B_{n,2} = \la y  \wh\YY_n \wh\YY_{n}^t \ra_{\g,w}.
$$
In particular, by \eqref{eq:ipd-parity}, (\ref{eq:xip1}) and (\ref{eq:xyip1}) it is easy to see that these matrices are of the form
\begin{align*}
 A_{1,1}  = \, \left[\begin{matrix} 0 & \ast & 0 \\ 0 & 0 & \ast  \end{matrix}\right],
  \quad A_{2m,1}  = & \, \left[\begin{matrix} \ast & 0 & 0 \\ 0 & 0 & 0 \\ 0 & 0 & \ast  \end{matrix}\right] 
\quad \hbox{and}\quad 
A_{2m+1,1}  =  \, \left[\begin{matrix} 0 & \ast & 0 \\ 0 & 0 & \ast \\ 0 & 0 & 0 \end{matrix}\right],\quad m \ge 1 \\
B_{1,1}  =  \, \left[\begin{matrix} \ast & 0 \\ 0 & \ast  \end{matrix}\right], \quad 
B_{2m,1} = & \, \left[\begin{matrix} \ast & \ast & 0 \\ \ast & \ast & 0 \\ 0 & 0 & \ast  \end{matrix}\right]
\quad \hbox{and}\quad 
B_{2m+1,1}  =  \, \left[\begin{matrix} \ast & 0 & 0 \\ 0 & \ast & \ast \\ 0 & \ast & \ast  \end{matrix}\right], \quad m \ge 1. 
\end{align*}
and 
\begin{align*}
 A_{1,2}  = \, \left[\begin{matrix} 0& 0 & \ast \\ \ast & \ast & 0  \end{matrix}\right],
  \quad A_{2m,2}  = & \, \left[\begin{matrix} 0 & \ast & \ast \\ 0 & 0 & \ast \\ \ast & 0 & 0  \end{matrix}\right] 
\quad \hbox{and}\quad 
A_{2m+1,2}  =  \, \left[\begin{matrix} 0 & 0 & \ast \\ \ast & \ast & 0 \\ 0 & \ast & 0  \end{matrix}\right],\quad m \ge 1 \\
B_{1,2}  =  \, \left[\begin{matrix} 0 & \ast \\ \ast & 0  \end{matrix}\right], \quad 
B_{2m,2} = & \, \left[\begin{matrix} 0& 0 & \ast \\ 0 & 0 &\ast \\ \ast & \ast & 0  \end{matrix}\right]
\quad \hbox{and}\quad 
B_{2m+1,2}  =  \, \left[\begin{matrix} 0 & \ast & \ast \\ \ast & 0 & 0 \\ \ast & 0 & 0 \end{matrix}\right], \quad m \ge 1. 
\end{align*}
These three-term relations in two variables hold when $(x,y)$ are on the cubic curve $\g$ or modulo  the 
polynomial ideal $\la y^2-\phi(x)\ra$. It is worth mentioning that we obtain, for example, 
$$
 x \wh Y_{2m,2} = (B_{2m,1})_{2,1} \wh Y_{2m,1} + (B_{2m,1})_{2,2} \wh Y_{2m,2} + (A_{2m-1,1})_{2,1} \wh Y_{2m-1,1},
$$
where $(A)_{i,j}$ stands for $(i,j)$-element of the matrix $A$, in which the lefthand side is a polynomial
of degree $n+1$, where the righthand side is of degree $n$. This holds without contradiction because 
 $(x,y) \in \g$.

\subsection{Examples of orthogonal polynomials on cubic curves}\label{sect:OPseg}
Recall from Theorem~\ref{thm:OPbasis} that we require two one-variable orthogonal polynomial families 
to construct an orthogonal basis on the cubic curve.
We shall always choose $w$ to be either the Jacobi or Laguerre weight. As we shall argue below the second family
of orthogonal polynomials  $ p_n(\phi w)$ will be non-classical if at least one of the roots 
of $\phi$ is not at an endpoint of $\mathrm{supp}(w)$. Otherwise, if all the roots of $\phi$ are at the endpoint(s)
of $\mathrm{supp}(w)$, the second orthogonal polynomial family will also be Jacobi or Laguerre but with different 
parameters for its weight function. 

The following cases arise:
\begin{itemize}
\item[1.] $\phi$ has one simple real root, or $\phi = (x-a)g_2(x)$, where $g_2(x)$ is a second degree polynomial 
with a pair of complex conjugate roots. 
\item[2.] $\phi$ has one triple real root, or $\phi = k(x-a)^3$, where $k$ is a constant. 
\item[3.] $\phi$ has a simple and a double real root, or $\phi = k(x -a)(x-b)^2$ where $a < b$ or $b < a$. 
The latter case is not of interest since then $b$ represents an isolated point of the curve $y^2 = \phi$.
\item[4.] $\phi$ has three distinct real roots, or $\phi = k(x-a)(x-b)(x-c)$, $a<b<c$.
\end{itemize}
In each of these cases the domain $\Omega_{\gamma}$ on which the orthogonal polynomials are defined is either a semi-infinite 
interval $[A, \infty)$ or the union of a compact interval $[a, b]$ and a semi-infinite interval.  On semi-infinite 
intervals, we additionally consider the cases for which the support of the weight is (i) also semi-infinite, 
or (ii) compact. Through a linear change of variables we may, without loss of generality, let 
$\mathrm{supp}(w) \subset \Omega_{\gamma}$ be the canonical compact interval $[-1, 1]$ or the semi-infinite 
interval $[0, \infty)$. If $\Omega_{\gamma}$ is the union of a compact and semi-infinite interval, we may 
consider $[-1, 1]$ and $[0, \infty)$ separately. If $\mathrm{supp}(w)$ includes the root(s) of $\phi$, then these will be mapped to $-1$ and/or $1$ if the domain is compact and to $0$ if the domain is semi-infinite. This implies that on $[-1, 1]$, if we let the weight $w$ be the 
Jacobi weight $w_{\alpha,\beta}$, then there exists a polynomial $g_k(x)$ of degree $k$, 
$0 \leq k \leq 3$, that is strictly positive on $\mathrm{supp}(w)$ and whose roots are outside $\mathrm{supp}(w)$,  such that 
\begin{align}
\phi =  w_{i,j}g_k, \qquad 0 \leq i, j \leq 3, \qquad 0\leq k \leq 3, \qquad i + j + k =3.   \label{eq:jacnonclassw}
\end{align}
Hence, $\phi$ has $i+j$ roots at the endpoints of $[-1, 1]$ and $k$ roots outside $[-1, 1]$. The weight 
function of the second orthogonal polynomial family is $\phi w = w_{\alpha+i,\beta+j}g_k$, which is a 
non-classical weight if $k>0$ and classical if $k=0$. On $[0, \infty)$, if we let $w$ be the Laguerre weight 
$w_{\alpha}$, then
\begin{align}
\phi = x^j g_k, \qquad
  0 \leq j \leq 3, \qquad 0 \leq k \leq 3, \qquad j + k =3.  \label{eq:legnonclassw}
\end{align}
Here $\phi$ has $j$ roots at the endpoint of $[0, \infty)$  and $k$ roots outside $[0, \infty)$. The weight function 
of the second family is $\phi w = w_{\alpha+j}g_k$, which is non-classical if $k>0$ and classical if $k=0$.


We consider three cubic curves as examples. For the first two, orthogonal polynomials can be given explicitly in terms 
of classical orthogonal polynomials. The third example discusses orthogonality on elliptic curves. 

\subsubsection{Orthogonal polynomials on the curve $y^2 = x^3$} 
In this example, the curve and the weight functions are 
$$
     y^2 = \phi = x^3 \quad \hbox{and} \quad w_\a(x) = x^\a e^{-x}, \qquad \a > -1.
$$
Since all the roots of $\phi$ are at $0$, the endpoint of $\mathrm{supp}(w)$, the orthogonal polynomial basis on the curve can be constructed entirely out of Laguerre polynomials. The polynomial $p_n(w;x) = L_n^{(\a)}(x)$ is the classical Laguerre polynomial of degree $n$,
$$
  L_n^{(\a)}(x) =\frac{(\a+1)_n}{n!} \sum_{k=0}^n \frac{(-n)_k}{(\a+1)_k } \frac{x^k}{k!}. 
$$
Moreover, $p_n(\phi w; x) = L_n^{(\a+3)}(x)$ is also an Laguerre polynomial with parameter $\a+3$.
In this setting the inner product on the curve becomes
\begin{align*}
   \la f,g\ra_{\g,w} \, &= \int_\g f(x,y) g(x,y) w_\a(x) d\s(x,y) \\
& = \int_0^\infty \left [ f(x,x^{3/2})g(x,x^{3/2}) + f(x, - x^{3/2})g(x, - x^{3/2})\right] w_\a(x) \d x.
\end{align*}
The orthogonal basis $\CB_n$ of the space $\CV_n$ in Theorem \ref{thm:OPbasis} becomes 
\begin{align*}
  \CB_{2m}  = & \left\{L_{3m}^{(\a)}(x), L_{3m-1}^{(\a)}(x), y L_{3m-2}^{(\a+3)}(x)\right\}, \\
  \CB_{2m+1} = & \left \{L_{3m+1}^{(\a)}(x), y L_{3m}^{(\a+3)}(x), y L_{3m-1}^{(\a+3)}(x)\right\}.  
\end{align*}
The norm of the Laguerre polynomial $L_n^{(\a)}$ is given by
$$
  h_n^{(\a)} = \frac{1}{\Gamma(\a+1)} \int_0^\infty [L_n^{(\a)}(x)]^2 e^{-x} \d x = \binom{n+\a}{n}, 
$$
from which the norm of the basis in $\CV_n$ can be derived as in Theorem \ref{thm:OPbasis}. 

\subsubsection{Jacobi polynomials on tear drop curves}
In this example, the curve is the tear drop curve 
\begin{equation}
   y^2= \phi= \tfrac14  (1-x)^2 (1+x), \qquad -1 \le x \le 1. 
\end{equation}
and the weight function is the Jacobi weight, for $\a, \b > -1$,
\begin{equation}
     w_{\a,\b}(x) = (1-x)^\a (1+x)^\b, \qquad -1 \le x \le 1. 
\end{equation}
Since all the roots of $\phi$ are at the endpoints of $\mathrm{supp}(w_{\a,\b})$, the orthogonal basis on the tear drop
curve can be constructed entirely out of Jacobi polynomials. The polynomial $p_n(w;x)$ is the usual Jacobi polynomial
$P_n^{(\a,\b)}(x)$ 
$$
  P_n^{(\a,\b)}(x) = \binom{n+\a}{n} {}_2F_1 \left(\begin{matrix} -n, n+\a+\b \\ \a+1 \end{matrix}; \frac{1-x}{2} \right),
$$
and $p_n(\phi w)$ is also a Jacobi polynomial, $p_n(\phi w) = P_n^{(\a+2,\b+1)}$. 
In this setting the inner product on the curve becomes
\begin{align*}
   \la f,g\ra_{\g,w} &= \int_\g f(x,y) g(x,y) d\ell(x,y)  \\
&= \int_{-1}^1 \left [ f\left(x,\sqrt{\phi(x)}\right)g\left(x,\sqrt{\phi(x)}\right) 
    +  f\left(x,-\sqrt{\phi(x)}\right)g\left(x,-\sqrt{\phi(x)}\right)\right] w_{\a,\b}\; \d x.
\end{align*}
The orthogonal basis $\CB_n$ of the space $\CV_n$ in Theorem \ref{thm:OPbasis} becomes 
\begin{align*}
  \CB_{2m} & = \left\{P_{3m}^{(\a,\b)}(x), P_{3m-1}^{(\a,\b)}(x), y P_{3m-2}^{(\a+2,\b+1)}(x)\right\}, \\
  \CB_{2m+1} & = \left \{P_{3m+1}^{(\a,\b)}(x), y P_{3m}^{(\a+2,\b+1)}(x), y P_{3m-1}^{(\a+2,\b+1)}(x)\right\}.  
\end{align*}

\subsubsection{Orthogonal polynomials on elliptic curves} We consider two elliptic curves. The first one is given by
$$
     y^2 = x^3 - 2 x + 4 = (x+2) \left((x -1)^2 +1\right) =: \phi(x),
$$
which has one component. We set $u = x+2$ so that $\phi = u\left((u -3)^2 +1\right)$ has a root at $u=0$ and we can choose the classical Laguerre weight
$$
   w_\a(u) = u^\a e^{-u}, \quad \a > -1, 
$$
defined for $u \ge 0$.
In this setting, the $\lbrace p_n(w) \rbrace$ are given by the Laguerre polynomials $L_n^{(\a)}$. Since $\phi(u)$ has two roots outside $\mathrm{supp}(w_{\a})$, the orthogonal polynomial family $\lbrace p_n(\phi w) \rbrace$ is non-classical and orthogonal with respect to
$$
\phi(u) w_\a(u) = ((u -3)^2+1) u^{\a+1} e^{-u}, \qquad u \ge 0.
$$
The inner product in this setting becomes
\begin{align*}
  \la f,g \ra_{\g,w_\a} = \int_{0}^\infty \left[f \left(u, \sqrt{\phi(u)} \right)g \left(u, \sqrt{\phi(u)} \right) +  f \left(u, - \sqrt{\phi(u)} \right)g \left(u, - \sqrt{\phi(u)} \right) \right] w_\a \d u.
\end{align*}   
If we choose a weight function $w$ that is supported on $x \in [-2, 2]$, say, then the inner product is defined on a finite
segment of $\g$. We set $x = 2u$, where $u \in [-1, 1]$ and choose the Jacobi weight $w = w_{\a,\b}(u)$. In this case, the $\lbrace p_n(w) \rbrace$ are given by the Jacobi polynomials $P_n^{(\a,\b)}$ and the $\lbrace p_n(\phi w) \rbrace$ are non-classical orthogonal polynomials (since $\phi$ has two roots outside $\mathrm{supp}(w)$) with respect to the weight
\begin{align*}
\phi(u) w_{\a,\b}(u) = 2\left((2u-1)^2 + 1  \right)w_{\a,\b+1}(u).
\end{align*}

Our second elliptic curve is given by 
$$
   y^2 = x^3 - 4 x = x (x^2-4),
$$
which has two components. The first one is a closed curved with $-2\le x \le 0$ and the second one is an open
curve defined for $x \ge 2$. On the first component we set $u = x + 1 \in [-1, 1]$ and choose $w = w_{\a,\b}(u)$ in which case $p_n(w;u) = P_n^{(\a,\b)}(u)$ and the polynomials $\lbrace p_n(\phi w;u) \rbrace$ are orthogonal with respect to $\phi w = w_{\a +1,\b +1}(u-3)$. On the second component we let $u = x -2$ and choose $w = w_{\a}(u)$ so that $p_n(w;u) = L_n^{(\a)}(u)$ and the polynomials $\lbrace p_n(\phi w;u) \rbrace$ are orthogonal with respect to $\phi w = w_{\a +1}(u+4)(u+2)$. 

\medskip
\noindent{\bf Remark}:
For a particularly convenient and flexible method for numerically computing these non-classical orthogonal polynomials $p_n(\phi w)$, based on the Lanczos algorithm, see~\cite{O4,Ol}. This is described in more detail in Section~\ref{sect:approximation}.

\section{Quadrature rules and polynomial interpolation}\label{sect:quadpoly}
\setcounter{equation}{0}

We consider quadrature rules and polynomials interpolation on the cubic curve. 

\subsection{Quadrature rules}
First we recall Gauss quadrature for a weight function 
$w$ defined on the real line. Let $\Pi_N^{(x)}$ denote the space of univariate polynomials of degree at most $N$
in $x$ variable. Let $x_{k,N}$, $1\le k \le N$, be the zeros of the orthogonal polynomial $p_N(w)$ of degree $N$. 
These zeros are the nodes of the $N$-point Gaussian quadrature rule, which is exact for polynomials of degree $2N-1$, 
$$
  \int_\RR f(x) w(x) \d x = \sum_{k=1}^N \lambda_{k,N} f(x_{k,N}), \qquad \forall f\in \Pi_{2N-1}^{(x)},
$$ 
where $\l_{k,N}$ are the Gaussian quadrature weights. Let $\g$ be a cubic curve. We
denote by $\Pi_n(\g)$ the space of polynomials of degree at most $n$ restricted to the curve $\g$. By 
Theorem \ref{thm:OPbasis}, 
\begin{equation}\label{eq:Pi-g}
\Pi_{2m+1}(\g) = \Pi_{3m+1}^{(x)} \cup y\Pi_{3m}^{(x)} \quad\hbox{and}\quad
   \Pi_{2m}(\g) = \Pi_{3m}^{(x)} \cup y \Pi_{3m-2}^{(x)}.
\end{equation}
From Proposition \ref{prop:dimVn} it follows 
$$
   \dim \Pi_0(\g) =1 \quad\hbox{and}\quad  \dim \Pi_n(\g) = 3 n, \quad n \ge 1.
$$


\begin{thm}\label{thm:Gaussian}
Let $\g$ be a cubic curve and $w(x)$ be a weight function on $\Omega_\g$. For $n = 2m$, let $N = N_n = 3m$ and for $n = 2m+1$, let $N = N_n = 3m+1$.  Let 
\begin{equation}\label{eq:quad1}
   I_n (f):=  \sum_{k=1}^{N} \l_{k,N} \left[f(x_{k,N},y_{k,N})+f(x_{k,N},-y_{k,N})\right], \qquad y_{k,N} = \sqrt{\phi(x_{k,N})}, 
\end{equation}
where $x_{k,N}$ are zeros of $p_{N}(w)$ and $\l_{k,N}$ are the corresponding weights of the $N$-point Gauss 
quadrature rule. Then 
\begin{equation}\label{eq:quad2}
  \int_\g f(x,y) w(x) \d \s(x,y) = I_n(f), \qquad \forall f\in \Pi_{2n-1}(\g). 
\end{equation}
\end{thm}

\begin{proof}
Since $\Pi_{2n-1}(\g) = \bigoplus_{k=0}^{2n-1} \CV_k(\g,w)$, we verify the quadrature rule for the basis of 
$\CV_k(\g,w)$ in Theorem \ref{thm:OPbasis} for $0 \le k \le 2n-1$. 
Since $Y_{2j,3}$, $Y_{2j+1,2}$ and 
$Y_{2k+1,3}$ contain a single factor $y$, both sides of \eqref{eq:quad2} are zero by \eqref{eq:ipd2}
and \eqref{eq:quad1}. Thus, we need to verify \eqref{eq:quad2} for 
\begin{align*}
  & \{Y_{2j,1}, Y_{2j,2}: 1 \le j \le n-1\} \cup \{Y_{2j+1,1}: 1\le j\le n-1\} \\
 & \quad =  \{p_{3j}(w), p_{3j-1}(w): 1 \le j \le n-1\} \cup \{p_{3j+1}(w): 1\le j\le n-1\} = \Pi^*.
\end{align*}
For $n = 2m$, the highest degree of $p_\ell(w)$ in the set is $3n-2 = 6m-2 < 2N_{2m}-1$, whereas 
for $n= 2m+1$, it is $3n-2 = 6m+1 = 2 N_{2m+1}-1$. For $f(x,y) = p_\ell(w;x)$, \eqref{eq:quad2} 
becomes, by \eqref{eq:ipd2}, 
$$
  \int_{\Omega_\g} p_\ell(w;x) w(x) \d x =  \sum_{k=1}^{N} \l_{k,N} p_\ell(w; x_{k,N}), 
$$
which holds, by the Gaussian quadrature, for $0 \le \ell \le 2N_n-1$. Consequently, it holds for all
polynomials in $\Pi^*$. This verifies \eqref{eq:quad2} for all polynomials in $\Pi_{2n-1}(\g)$ and 
completes the proof. 
\end{proof}
 
\begin{rem} \label{rem:gaussCF}
For $n = 2m+1$, the quadrature \ref{eq:quad2} uses $2 N_{2m+1} = 6m + 2 = 3 n-1$, whereas for  $n = 2m$, 
the quadrature \ref{eq:quad2} uses $2 N_{2m} = 6m = 3 n$ points. It is exact for the space $\Pi_{2n-1}(\g)$, 
which has the dimension $3 (2n-1) = 6 n - 3$. 
\end{rem} 
 
The quadrature \eqref{eq:quad1} on the cubic curve is an analogue of the Gaussian quadrature rule on the
real line. We now consider polynomial interpolation based on the nodes of this quadrature rule. 

\subsection{Lagrange interpolation}

First we recall the univariate
Lagrange interpolation polynomial on the zeros $x_{k,N}$, $1 \le k \le N$, of $p_N(w)$, denoted 
by $L_N (w; f)$, which is the unique polynomial of degree at most $N-1$ that satisfies
$$
   L_N (w; f, x_{k,N}) = f(x_{k,N}), \qquad 1 \le k \le N, 
$$
for any continuous function $f$. It is well-known that $L_N(w; f)$ is given by
\begin{equation}\label{eq:LagrangeInte1}
  L_N (w;f,x) = \sum_{k=1}^N f(x_{k,N})\ell_k(x), \qquad \ell_k(x) = \frac{p_N(w;x)}{(x-x_{k,N}) p_N'(w_{k,N})}.
\end{equation}
By the Christoffel--Darboux formula, we can also write $\ell_k$ as
\begin{equation}\label{eq:LagrangeInte2}
  \ell_k(x) = \frac{K_N(w; x,x_{k,N})}{K_N(w; x_{k,N},x_{k,N})}, 
  \qquad K_N(x,y) = \sum_{k=0}^{N-1} \frac{p_k(w;x)p_k(w;y)}{h_k(w)}. 
\end{equation}


\begin{thm}\label{thm:laginterp}
Let $\g$ be a cubic curve and $w(x)$ be a weight function on $\Omega_\g$. For $f \in C(\Omega_\g)$, let 
$f_o$ and $f_e$ be defined as in \eqref{eq:fe-fo}. For $n = 0,1,2\ldots$, let $N = N_n$ be defined by
$N_{2m} = 3m$ and $N_{2m+1} = 3m+1$. Let 
\begin{equation}\label{eq:interp}
   \CL_n (w; f, x,y):=  L_{N} (w; f_e, x) + y L_{N} (w; f_o, x),
\end{equation}
Then $\CL_n(w;f)$ is a polynomial that satisfies 
\begin{align*}
\begin{split}
  \CL_n(w; f, x_{k,N}, y_{k,N})  &\, = f(x_{k,N},y_{k,N}), \\ 
  \CL_n(w; f, x_{k,N}, - y_{k,N}) &\, = f(x_{k,N},- y_{k,N}),
\end{split}
\qquad 1 \le k\le N.
\end{align*}  
Furthermore, it is the unique interpolation polynomial in $\Pi_n(\g)$ if $n = 2m+1$ and in
$\Pi_{n}(\g)\cup\{ Y_{2m+1,3}\}$ if $n = 2m$. 
\end{thm}

\begin{proof}
If $n = 2m+1$, then $N =3m+1$ and  we interpolate at $2N = 6m+2$ points.  In this case, both $L_N(w;f_e)$ and $L_N(w;f_o)$ 
are elements of $\Pi_{3m}^{(x)}$. Since $\Pi_{2m+1}(\g) = \Pi_{3m+1}^{(x)} \cup y\Pi_{3m}^{(x)}$, it follows that 
$\CL_n(w;f) \in \Pi_{2m+1}(\g)$. Since $Y_{2m+1,1}(x,y) = p_{3m+1}(w;x)$ vanishes on all interpolation
points, there are $\dim \Pi_{2m+1}(\g) - 1 = 6m+2$ independent functions over the set of nodes in the space 
$\Pi_{2m+1}(\g)$. Similarly, if $n=2m$, then $N = 3m$ and we interpolate at $2N = 6m$ points. In this case, 
both $L_N(w;f_e)$ and $L_N(w;f_o)$ are elements of $\Pi_{3m-1}^{(x)}$. Since $\Pi_{2m}(\g) = \Pi_{3m}^{(x)} 
\cup y \Pi_{3m-2}^{(x)}$, it follows that $\CL_n(w;f) \in \Pi_{2m}(\g) \cup \{Y_{2m+1,3}\}$ since 
$Y_{2m+1,3} = y p_{3m-1}( w;x)$. 
Since $Y_{2m,1}(x,y) =p_{3m}(w;x)$ vanishes at all nodes, we see 
that there are $\dim \Pi_{2m}(\g) - 1 + 1= 6m$ independent functions over the set of nodes in the space. Now, 
for $1\le k \le N$, we obtain from the Lagrange interpolation of $L_N(w; f)$ that 
\begin{align*}
 \CL_n(w; f, x_{k,N}, y_{k,N}) &\, =  L_N (w; f_e, x_{k,N}) + y_{k,N} L_{N} (w; f_o, x_{k,N}) \\
      & \, = f_e(x_{k,N}) + y_{k,N} f_o(x_{k,N}) = f(x_{k,N}, y_{k,N}); 
\end{align*}
similarly, we also obtain that 
\begin{align*}
 \CL_n(w; f, x_{k,N}, - y_{k,N}) = f_e(x_{k,N}) - y_{k,N} f_o(x_{k,N}) = f(x_{k,N}, - y_{k,N}), 
\end{align*}
so that $\CL_n(w;f)$ satisfies the desired interpolation conditions. 

Finally, since zeros of $p_N(w)$ are all in the interior of $\Omega_\g$, it follows that 
$y_{k,N} = \sqrt{\phi(x_{k,N})} > 0$ for all $k$. Consequently, if $f(x_{k,N}, y_{k,N}) =0$ for all $1 \le k \le N$, 
then $ L_N (w; f_e, x_{k,N}) =0$ and $ L_N (\phi w; f_o, x_{k,N}) =0$ for all $k$, so that, by the uniqueness
of the Lagrange interpolation, $f_e = 0$ and $f_o = 0$. Consequently, $f =0$, which proves that the interpolation
polynomials are unique in their respective spaces.  
\end{proof}

\subsection{Interpolation via quadrature}\label{sect:interpviaquadsect}
For computational purposes, it is convenient to express the interpolant defined in Theorem~\ref{thm:laginterp} as a truncated expansion in the orthogonal polynomial basis on $\g$. We recall a classical result for univariate interpolants, according to which (\ref{eq:LagrangeInte1}) can be expressed as an expansion in the orthogonal polynomials $p_n(w)$,
\begin{align}
L_N(w;f,x) = \sum_{k=0}^{N-1} a_{k,N}p_k(w;x), \label{eq:uniexp}
\end{align}
where
\begin{align*}
a_{k,N} = \frac{\la f, p_k(w) \ra_N}{\la p_k(w), p_k(w) \ra_N}  = \frac{\la f, p_k(w) \ra_N}{h_k(w)}, \qquad 0 \leq k \leq N-1,
\end{align*}
and $\la \cdot, \cdot \ra_{N}$ denotes the discretised univariate inner product $\la \cdot, \cdot \ra_{w}$ based on the $N$-point Gaussian quadrature rule:
\begin{align*}
  \la f(x),g(x)\ra_N :=  \sum_{k=1}^N \l_{k,N} f(x_{k,N})g (x_{k,N}),  
\end{align*}
where, as before, $x_{k,N}$ and  $\l_{k,N}$ are the Gaussian quadrature nodes (the roots of $p_{N}(w)$) and weights, respectively. 

According to the following bivariate analogue of the univariate result just mentioned, we require a system of $M$ orthogonal functions with respect to an $M$-point discrete inner product to construct an interpolant expanded in an orthogonal basis.
\begin{prop}\cite{OX2} \label{prop:intviaquad}
Suppose we have a discrete inner product for a basis $\lbrace \phi_j \rbrace_{j=0}^{M-1}$ of the form
\begin{align*}
\la f, g \ra_{M} = \sum_{j = 1}^M w_j f(x_j,y_j)g(x_j,y_j)
\end{align*}
satisfying $\la \phi_m, \phi_n \ra_{M} = 0$ for $m \neq n$ and $\la \phi_n, \phi_n \ra_{M} \neq 0$. Then the function
\begin{align*}
\CL_M f(x,y) = \sum_{n = 0}^{M-1} f_{n}^{M}\phi_n(x,y)
\end{align*}
interpolates $f(x,y)$ at $(x_j,y_j)$, where $f_n^M := \frac{\la \phi_n, f \ra_M}{\la \phi_n, \phi_n \ra_{M}}$.
\end{prop}
We first consider the inner product  $\la \cdot, \cdot \ra_{N,\g}$ coming from discretization of 
$\la \cdot, \cdot \ra_{\g,w}$ via Gauss quadrature,
\begin{align}
 \la f,g \ra_{N,\g} :=\,\sum_{k=1}^N \lambda_{k,N}\left[ f(x_{k,N},y_{k,N})g(x_{k,N},y_{k,N}) + f(x_{k,N},-y_{k,N})g(x_{k,N},-y_{k,N}) \right],  \label{eq:Gip}
\end{align}
where $x_{k,N}$ and $\lambda_{k,N}$ are the $N$-point Gauss quadrature nodes and weights and $y_{k,N} = \sqrt{\phi(x_{k,N})}$. From Proposition~\ref{prop:intviaquad}, we require a system of $2N$ orthogonal functions with respect to $ \la f,g \ra_{N,\g}$ to construct an interpolant expanded in an orthogonal basis. The next result shows that only $2N-1$ functions from the orthogonal basis on $\g$ given in Theorem~\ref{thm:OPbasis}  are orthogonal with respect to $ \la f,g \ra_{N,\g}$.
\begin{prop}
With $n = 2m+1$ and $N = N_n = 3m + 1$, the $2N-1$ functions
\begin{align}
\left\lbrace Y_0, Y_{1,1}, Y_{1,2} \right\rbrace \cup \left\lbrace  Y_{k,1}, Y_{k,2}, Y_{k,3}  \right\rbrace_{k=2}^{n-1}\cup \lbrace Y_{n,3} \rbrace = \left\lbrace p_k(w), yp_k(\phi w)  \right\rbrace_{k = 0}^{3m-1}\cup \left\lbrace  p_{3m}(w)  \right\rbrace, \label{eq:oddsetg}
\end{align}
and with $n = 2m$ and $N = N_n = 3m$, the $2N-1$ functions
\begin{align}
\left\lbrace Y_0, Y_{1,1}, Y_{1,2} \right\rbrace \cup \left\lbrace  Y_{k,1}, Y_{k,2}, Y_{k,3}  \right\rbrace_{k=2}^{n} \setminus \lbrace Y_{n,1} \rbrace = \left\lbrace p_k(w), yp_k(\phi w)  \right\rbrace_{k = 0}^{3m-2}\cup \left\lbrace  p_{3m-1}(w)  \right\rbrace   \label{eq:evensetg}
\end{align}
are the largest sets of functions from among the $Y_{k,i}$ defined in Theorem~\ref{thm:OPbasis} that 
are orthogonal and have  nonzero norms with respect to $\la\cdot,\cdot\ra_{N,\g}$. 
\end{prop}
\begin{proof}
From the symmetry of $\la\cdot,\cdot\ra_{N,\g}$, we have, similar to the property (\ref{eq:ipd-parity}) of the continuous inner product,
\begin{align*}
\la f(x),yg(x) \ra_{N,\g} = 0.
\end{align*}
This proves the orthogonality between the basis functions $Y_{k,i}$ of the form $f(x)$ and those of the form $yg(x)$. To demonstrate orthogonality between the remaining $Y_{k,i}$ (those that are both of the from $f(x)$ or both of the form $yf(x)$), we note that since $N$-point Gaussian quadrature is exact for polynomials of degree $\leq 2N - 1$,
\begin{align}
\la p_k(w),p_j(w) \ra_{N,\g} = 2\la p_k(w),p_j(w) \ra_{N} = 2\la p_k(w),p_j(w) \ra_{w}=2\delta_{k,j}h_k(w),  \label{eq:2to1gx}
\end{align}
for $0 \leq k, j \leq N - 1$. Since $\phi$ has degree 3,
\begin{align} \label{eq:2to1gyx}
 \la yp_k(\phi w),yp_j(\phi w) \ra_{N,\g} \,& = 2\la \phi p_k(\phi w),p_j(\phi w) \ra_{N} \\
      & =  2\la p_k(\phi w),p_j(\phi w) \ra_{\phi w} = 2\delta_{k,j}h_k(\phi w), \notag
\end{align}
for $0 \leq k,j \leq N - 2$.  Equations (\ref{eq:2to1gx}) and (\ref{eq:2to1gyx}) demonstrate the orthogonality of the sets of functions in (\ref{eq:evensetg}) and (\ref{eq:oddsetg}) as well as their nonzero norms. The aforementioned bounds on the indices $k$ and $j$ and the fact that  $p_N(w)$ vanishes at all the quadrature nodes imply that it is not possible  to add a function $p_k(w)$, with $k \geq N$ or $yp_k(\phi w)$, with $k \geq N-1$ to the sets (\ref{eq:evensetg}) or (\ref{eq:oddsetg}) such that it is orthogonal and have a  nonzero norm with respect to $\la\cdot,\cdot\ra_{N,\g}$. In particular, it is not possible to add $yp_{N-1}$ to  (\ref{eq:evensetg}) or (\ref{eq:oddsetg}) since $\la yp_{N-1},yp_{N-2}\ra_{N,\g} = 2\la \phi p_{N-1},p_{N-2}\ra_{N} \neq 0$.
\end{proof}
We conclude that it is not possible to construct an interpolant via Gauss quadrature in the manner of Proposition~\ref{prop:intviaquad}. It is nevertheless possible to construct an interpolant \`a la Proposition~\ref{prop:intviaquad} if $\la\cdot, \cdot \ra_{\g,w}$ is discretised via Gauss--Radau or Gauss--Lobatto quadrature~\cite{gautschi}. 

As shown in the proof of Theorem~\ref{thm:laginterp}, the set of $2N$ functions $\lbrace p_k(w), yp_n(\phi w) \rbrace_{k = 0}^{N-1}$ at the $2N$ nodes $(x_{k,N},\pm y_{k,N})$, $k = 1, \ldots, N$ are linearly independent. Hence, the interpolant in Theorem~\ref{thm:laginterp} can be represented in the form
\begin{align}
\CL_n (w; f, x,y) = \sum_{k = 0}^{N-1} a_{k,N} p_k(w;x) + y\sum_{k = 0}^{N-1} b_{k,N} p_k(\phi w;x),   \label{eq:b1interpexp}
\end{align}
and the coefficients $a_{k,N}$ and $b_{k,N}$ can be obtained by solving a $2N \times 2N$ Vandermonde-like linear system whose columns consist of $p_k(w)$ and $yp_n(\phi w)$, evaluated at $(x_{k,N},\pm y_{k,N})$ for $k = 1, \ldots, N$. However, since this procedure requires $\mathcal{O}(N^3)$ operations, we introduce an alternative method with which the coefficients can be obtained with a complexity of either  $\mathcal{O}(N^2)$ (for a general weight $w$) or $\mathcal{O}(N \log N)$ (for the Chebyshev weight).
 
We now introduce a new inner product and orthogonal basis with respect to which it is possible to construct an interpolant with Gaussian quadrature. Although the inner product $\la\cdot, \cdot \ra_{\g,w}$ is the natural parametrization of (\ref{eq:ipd1}), the new inner product gives rise to a new orthogonal basis that is more efficient for computational purposes compared to the basis in Theorem~\ref{thm:OPbasis}, as we shall demonstrate.

 We define the inner product $\left[ \cdot,\cdot\right]_{\g,w}$ as 
\begin{align} \label{eq:ipd3}
  \left[ f,g \right]_{\g,w} := \int_{\Omega_\g} & \left[  f_e\left(x \right) g_e\left(x \right) + f_o\left(x \right) g_o\left(x \right)\right] w(x) \d x ,
\end{align}
where $f_e, g_e, f_o, g_o$ are the even and odd parts, respectively, of functions $f(x,y)$ and $g(x,y)$ on $\g$, as defined in (\ref{eq:fe-fo}). Note that the odd part of functions on $\g$ have removable singularities at points $(x,y)$ where $y = \sqrt{\phi(x)} = 0$.

In the inner product space defined by $\left[ \cdot,\cdot\right]_{\g,w}$, the orthogonal polynomial basis on $\g$ is exactly the same as in Theorem~\ref{thm:OPbasis} (where the inner product space is equipped with $\la \cdot,\cdot\ra_{\g,w}$), except that the $Y_{n,i}$ that are of the form $yp_k(\phi w)$ are replaced by $yp_k( w)$. That is, whereas the orthogonal polynomial basis in Theorem~\ref{thm:OPbasis} is constructed from $p_n(w)$ and $p_n(\phi w)$, the orthogonal basis with respect to $\left[ \cdot,\cdot\right]_{\g,w}$ is constructed solely out of $p_n(w)$. Hence, in the inner product space with $\left[ \cdot,\cdot\right]_{\g,w}$, one only requires univariate classical orthogonal polynomials.

The results in Theorem~\ref{thm:partsums} and Corollary~\ref{cor:psumconv} also hold for the orthogonal basis with respect to  $\left[ \cdot,\cdot\right]_{\g,w}$, mutatis mutandis. A notable difference between the inner products $\la\cdot, \cdot \ra_{\g,w}$ and $\left[ \cdot,\cdot\right]_{\g,w}$ is that, unlike the Jacobi operators for $\la\cdot, \cdot \ra_{\g,w}$ in section~\ref{sect:jacops}, the latter gives rise to a non-symmetric Jacobi operator for multiplication by $y$. This is because the inner product $\left[ \cdot, \cdot \right]_{\g,w}$ is not self-adjoint with respect to multiplication by $y$. That is,
 \begin{align*}
\left[ yf(x), yg(x) \right]_{\g,w} \neq \left[ y^2 f(x), g(x) \right]_{\g,w},
\end{align*}
because $\left[ yf(x), yg(x) \right]_{\g,w} = \la f(x), g(x) \ra_{w}$ and $\left[ y^2 f(x), g(x) \right]_{\g,w} = \left[ \phi f(x), g(x) \right]_{\g,w} = \la \phi f(x), g(x) \ra_{w}$. 

\subsection{Interpolation via quadrature with respect to $\left[ \cdot,\cdot\right]_{\g,w}$}\label{sect:interpvquad2}
Discretising $\left[\cdot, \cdot \right]_{\g,w}$ via Gauss quadrature, we obtain
\begin{align*}
  [f,g]_N :=  \sum_{k=1}^N \l_{k,N} \left[ f_e(x_{k,N})g_e (x_{k,N}) + 
          f_o(x_{k,N})g_o (x_{k,N})\right].  
\end{align*}
We shall need the following result.
\begin{prop}\label{prop:discorth}
With $n = 2m+1$ and $N = N_n = 3m + 1$, the $2N$ functions
\begin{align}
\left\lbrace Y_0, Y_{1,1}, Y_{1,2} \right\rbrace \cup \left\lbrace  Y_{k,1}, Y_{k,2}, Y_{k,3}  \right\rbrace_{k=2}^n\setminus \lbrace Y_{n,1} \rbrace = \left\lbrace p_k(w), yp_k(w)  \right\rbrace_{k = 0}^{3m}, \label{eq:oddset}
\end{align}
and with $n = 2m$ and $N = N_n = 3m$, the $2N$ functions
\begin{align}
\left\lbrace Y_0, Y_{1,1}, Y_{1,2} \right\rbrace \cup \left\lbrace  Y_{k,1}, Y_{k,2}, Y_{k,3}  \right\rbrace_{k=2}^n\cup \lbrace Y_{n+1,3} \rbrace \setminus \lbrace Y_{n,1} \rbrace = \left\lbrace p_k(w), yp_k(w)  \right\rbrace_{k = 0}^{3m-1}   \label{eq:evenset}
\end{align}
are orthogonal and have nonzero norms  with respect to $[\cdot,\cdot]_N$.
\end{prop}
\begin{proof}
It follows from the definition (\ref{eq:fe-fo}) that for a function on $\g$ that depends only on $x$,
\begin{align*}
\left( f(x) \right)_e = f(x), \qquad \left( f(x) \right)_o = 0, 
\end{align*}
and for a function on $\g$ of the form $yg(x)$,
\begin{align*}
\left( yg(x) \right)_e = 0, \qquad \left( yg(x)  \right)_o = g(x). 
\end{align*}
Hence, $\left[ f(x), y g(x) \right]_{N} = 0$, which proves orthogonality between functions in the sets (\ref{eq:evenset}) and (\ref{eq:oddset}) of the form $p_k(w)$ and those of the form $yp_k(w)$.
We note that since the $N$-point Gaussian quadrature rule is exact for polynomials of degree $\leq 2N-1$, it follows that
\begin{align*}
 [p_k(w), p_j(w)]_N \, &= \la p_k(w), p_j(w)\ra_N =\la p_k(w), p_j(w)\ra_w=\delta_{k,j}h_k(w), 
 \\
 [y p_k(w), y p_j(w)]_N \, & = \la p_k(w), p_j(w) \ra_N  = \la p_k(w), p_j(w)\ra_w=\delta_{k,j}h_k(w),
\end{align*}
for $0 \le k, j \le  N-1$.
This demonstrates both the orthogonality between functions in (\ref{eq:evenset}) and (\ref{eq:oddset}) that are both of the form $f(x)$ or both of the form $yf(x)$ and the nonzero norms of the functions in (\ref{eq:evenset}) and (\ref{eq:oddset}) with respect to $[\cdot,\cdot]_N$.  
\end{proof}
\begin{thm}\label{thm:interpviaquadgauss}
The interpolant defined in Theorem~\ref{thm:laginterp} can be represented in the form
\begin{align}
\CL_n (w; f, x,y) = \sum_{k = 0}^{N-1} a_{k,N} p_k(w;x) + y\sum_{k = 0}^{N-1} b_{k,N} p_k(w;x),  \label{eq:interpexp}
\end{align}
with
\begin{align*}
a_{k,N} =  \frac{\la  f_e, p_k(w) \ra_{N}}{h_k(w)}, \qquad b_{k,N} =   \frac{\la  f_o, p_k(w) \ra_{N}}{h_k(w)}, \qquad 0 \leq k \leq N-1,
\end{align*}
and where $N = N_n = 3m+1$ if $n = 2m+1$ and $N = N_n = 3m$ if $n = 2m$.
\end{thm}
\begin{proof}
It follows from Proposition~\ref{prop:discorth} that for $0 \leq k \leq N-1$,
\begin{align*}
 \frac{\left[  f, p_k(w) \right]_{N}}{\left[ p_k(w), p_k(w) \right]_{N}} =  \frac{\la  f_e, p_k(w) \ra_{N}}{\la p_k(w), p_k(w) \ra_{N}} =  \frac{\la  f_e, p_k(w) \ra_{N}}{h_k(w)} = a_{k,N} ,
\end{align*}
and
\begin{align*}
 \frac{\left[  f, yp_k(w) \right]_{N}}{\left[ yp_k(w), yp_k(w) \right]_{N}} =  \frac{\la  f_o, p_k(w) \ra_{N}}{\la p_k(w), p_k(w) \ra_{N}} =  \frac{\la  f_o, p_k(w) \ra_{N}}{h_k(w)} = b_{k,N}. 
\end{align*}
By the previously mentioned result for univariate interpolants in section~\ref{sect:interpviaquadsect}, this implies that 
\begin{align}
\sum_{k = 0}^{N-1} a_{k,N} p_k(w;x) = L_N(w;f_e,x), \qquad \sum_{k = 0}^{N-1} b_{k,N} p_k(w;x) = L_N(w;f_o,x).   \label{eq:eointerp} 
\end{align}
Comparing (\ref{eq:eointerp}), (\ref{eq:interpexp}) and (\ref{eq:interp}), the result follows.
\end{proof}
Note that if $w$ is chosen to be the Chebyshev weight, one can compute the coefficients of the interpolant, $a_{k,N}$ and $b_{k,N}$, in $\mathcal{O}(N \log N)$ operations  using the fast cosine transform. A fast transform is also available for the uniform Legendre weight $w = 1$~\cite{HT16}, and potentially for any Jacobi weight~\cite{HT}. Since fast transforms are not available for the non-classical polynomials $p_n(\phi w)$, $\mathcal{O}(N^2)$ operations are required to compute the coefficients of the interpolant obtained via (Gauss--Radau or Gauss--Lobatto) quadrature. The inner product $\left[ \cdot, \cdot \right]_{\g,w}$ is therefore preferable to $\la\cdot, \cdot \ra_{\g,w}$ for computing interpolants via quadrature.

\section{Applications}\label{sect:approximation}
\setcounter{equation}{0}

We can approximate functions of the form 
\begin{align}
g(t) := f\left(t,\sqrt{\phi(t)}\right),  \label{eq:singforward}
\end{align}
which have square root-type singularities at the zero(s) of $\phi(t)$, by recasting them as functions $f(x,y)$ on the cubic curve $\gamma = \left\lbrace (x,y) \; : \; y^2 = \phi(x) \right\rbrace$. If $f(x,y)$ is a smooth function of $x$ and $y$ on $\gamma$, then the bivariate interpolant on $\gamma$ will converge much faster compared to the univariate interpolant of $g(t)$. Similarly, if $\phi$ has an inverse $\phi^{-1}$ on an interval, then we can approximate functions of the form
\begin{align}
f\left(\phi^{-1}(t^2),t\right),  \label{eq:singinverse}
\end{align}
which have cubic-type singularities where $\phi^{-1}$ has zeros, by an interpolant $f(x,y)$ on $\gamma$ by setting $y = t$ and $x = \phi^{-1}(t^2)$. First we approximate a function of the form (\ref{eq:singinverse}) and in the next section we consider functions of the form (\ref{eq:singforward}) that arise as solutions to differential equations.

To compute $p_k(w)$ (and $p_k(\phi w)$) we use a variant of the Stieltjes procedure where orthogonal polynomials are calculated via computing the connection coefficients with Legendre polynomials, that is, it computes the expansion
$$
p_k(w) = \sum_{j=0}^k c_{kj} \tilde P_k
$$
where $\tilde P_k$ are normalized Legendre polynomials. Provided $w$ is a polynomial then the inner products are computable exactly via the Legendre Jacobi operator $J$, that is,
$$
\langle \tilde P_k, \tilde P_j\rangle_w = {\bf e}_k^\top w(J) {\bf e}_j.
$$
Therefore we can readily determine $c_{kj}$ by orthogonalizing via Gram--Schmidt\footnote{Equivalently, this can be viewed as Lanczos iteration with a non-standard, banded inner product as explained in \cite{Ol}. A convenient implementation is available in the Julia package  OrthogonalPolynomialsQuasi.jl~\cite{O4}}.  This representation makes differentiation straightforward as we know $P_k'(x) = (k+1)/2 P_{k-1}^{(1,1)}(x)$ \cite[(18.9.15)]{DLMF}.  In practice, we take the Legendre weight $w(x) = 1$ so that $p_k(w) = \tilde P_k$ and we need only calculate $p_k(\phi)$.

\subsection{Function approximation}
To approximate the function
\begin{align}
g(t) =  J_1(10t + 20\sqrt[3]{t^2 + \epsilon^2}), \qquad t \in [-1, 1], \label{eq:funapproxex1}
\end{align}
where $J$ denotes the Bessel function of the first kind, we can set  $y^2 = \phi(x) = x^3 - \epsilon^2$, hence $\phi^{-1}(y^2) = \sqrt[3]{y^2 + \epsilon^2}$. Then
\begin{align*}
g(t) = f(x,y) = J_1(10y + 20x),
\end{align*} 
which is defined on $\gamma$, where
\begin{align}
\gamma = \left\lbrace (x,y) \; : \; y^2 = x^3 - \epsilon^2, \: y \in [-1, 1], \: x \in [\epsilon^{2/3}, (1 + \epsilon^2)^{1/3}] \right\rbrace.  \label{eq:cubiccurveex1}
\end{align}

For comparison purposes with standard bases, we also approximate (\ref{eq:funapproxex1}) using algebraic Hermite--Pad\'e (HP) approximation~\cite{QP}. Given the function values $f(x_{k,N})$, $1 \leq k \leq N$, $x_{k,N} \in [a, b]$, to find the HP approximant of $f$ on $[a, b]$, we require polynomials  $p_0, \ldots, p_m$ on $[a,b]$ of specified degrees $d_0, \ldots, d_m$ such that 
\begin{equation}
\| p_0 + p_1 f + p_2 f^2 + \cdots + p_m f^m \|_{N} = \text{minimum}. \label{eq:minprob}
\end{equation}
Here, the norm $\| \cdot \|_N^2   := \langle \cdot, \cdot \rangle_{N}$ is induced by the following discrete inner product
\begin{align}
\langle f, g \rangle_{N} = \sum_{k = 1}^{N} w_k f(x_{k,N}) g(x_{k,N}).  \label{eq:HPdiscinnerprod}
\end{align} 
We assume some kind of normalization so that the trivial solution $p_0 = \ldots = p_m = 0$ is not admissible. The HP approximant of $f(x)$, viz.~$\psi(x)$, is the algebraic function defined by
\begin{equation}
p_0(x) + p_1(x) \psi(x) + p_2(x) \psi^2(x) + \cdots + p_m(x) \psi^m(x) = 0. \label{eq:psidef}
\end{equation} 
In practice, we compute the polynomials $p_0, \ldots, p_m$ by expanding them in an orthonormal polynomial basis with respect to the discrete inner product (\ref{eq:HPdiscinnerprod}). Then (\ref{eq:minprob}) reduces to a least squares problem whose solution we compute with the SVD. Hence the implicit normalization used is that the vector of polynomial coefficients of $p_0, \ldots, p_m$ in the orthonormal basis is a unit vector. This computational approach is similar to that used in~\cite{gonnet, pachon} for the case $m = 1$, which corresponds to rational interpolation or least squares fitting. 

Throughout we shall consider diagonal HP approximants for which the degrees of the polynomials $p_0, \ldots, p_m$ are equal, say degree $d$. We require that the number of points at which $f$ is sampled is greater than or equals to the number of unknown polynomial coefficients:
\begin{align*}
N \geq m(d+1) + d.
\end{align*}
If $N = m(d+1) + d$, then the minimum attained by the solution to the least squares problem (\ref{eq:minprob}) is zero
\begin{equation}
N = m(d+1)+d, \quad \Rightarrow \quad \| p_0 + p_1 f + p_2f^2 + \cdots + p_mf^m \|_{N} = \text{minimum} = 0. \label{eq:HPinterp}
\end{equation}
We call this the interpolation case. If $N > m(d+1) + d$, then the minimum attained by the least squares solution will be nonzero in general. Throughout we shall approximate functions with HP interpolants.

Note that if $m=1$ and $p_1(x) = 1$ in (\ref{eq:HPinterp}), then the HP approximant, $\psi(x) = -p_0(x)$, is a polynomial interpolant of $f$ on the grid; if $m=1$, then the HP approximant, $\psi = -p_0(x)/p_1(x)$, is a rational interpolant of $f$ (with poles in the complex $x$-plane). If $m \geq 2$, then for every $x$, $\psi(x)$ will generally be an $m$-valued approximant of $f$ (with poles and algebraic branch points in the complex $x$-plane). We want to pick only one branch of the $m$-valued function $\psi$ to approximate $f$. One way to do this is to solve (\ref{eq:psidef}) with Newton's method using a polynomial or rational approximant as first guess. 

Figure~\ref{fig:funapproxex1} compares the rate of convergence to $g(t)$ in (\ref{eq:funapproxex1}) of HP approximants with $m=0, 1, 2, 3$ (polynomial, rational, quadratic and cubic HP interpolants) and interpolants on the cubic curve (\ref{eq:cubiccurveex1}) which were obtained via quadrature in the Chebyshev basis as described in section~\ref{sect:interpvquad2}. The figure shows that the interpolant on the cubic curve $\gamma$ converges super-exponentially (since $f$ is an entire function in $x$ and $y$), which is significantly faster than HP approximants (which in addition appear to have stability/ill-conditioning issues). Moreover, it is robust as $\epsilon \rightarrow 0$, whereas HP interpolants break down in accuracy.

\begin{figure}[ht]
  \begin{center}
    \includegraphics[width=0.495\textwidth]{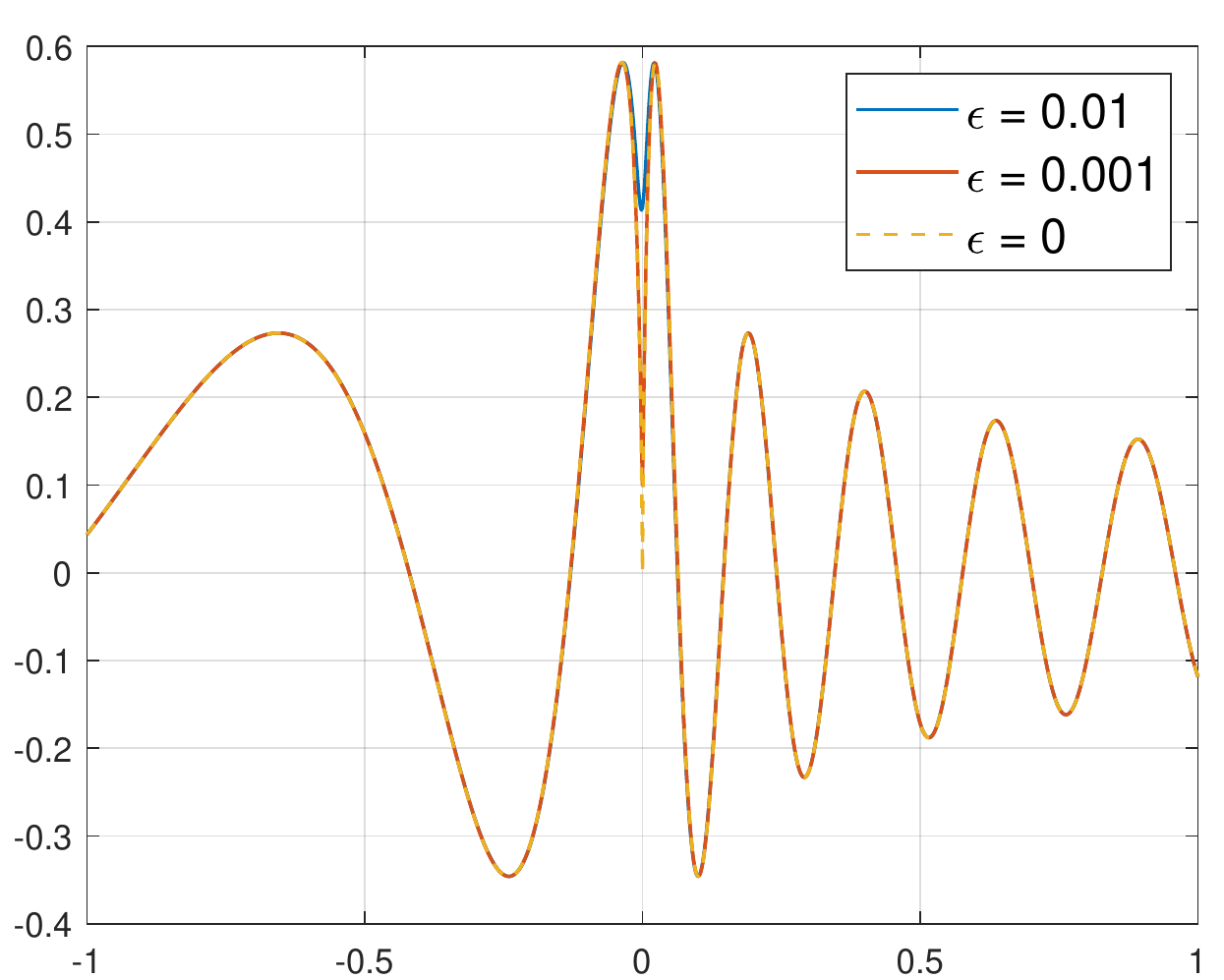}
    \includegraphics[width=0.495\textwidth]{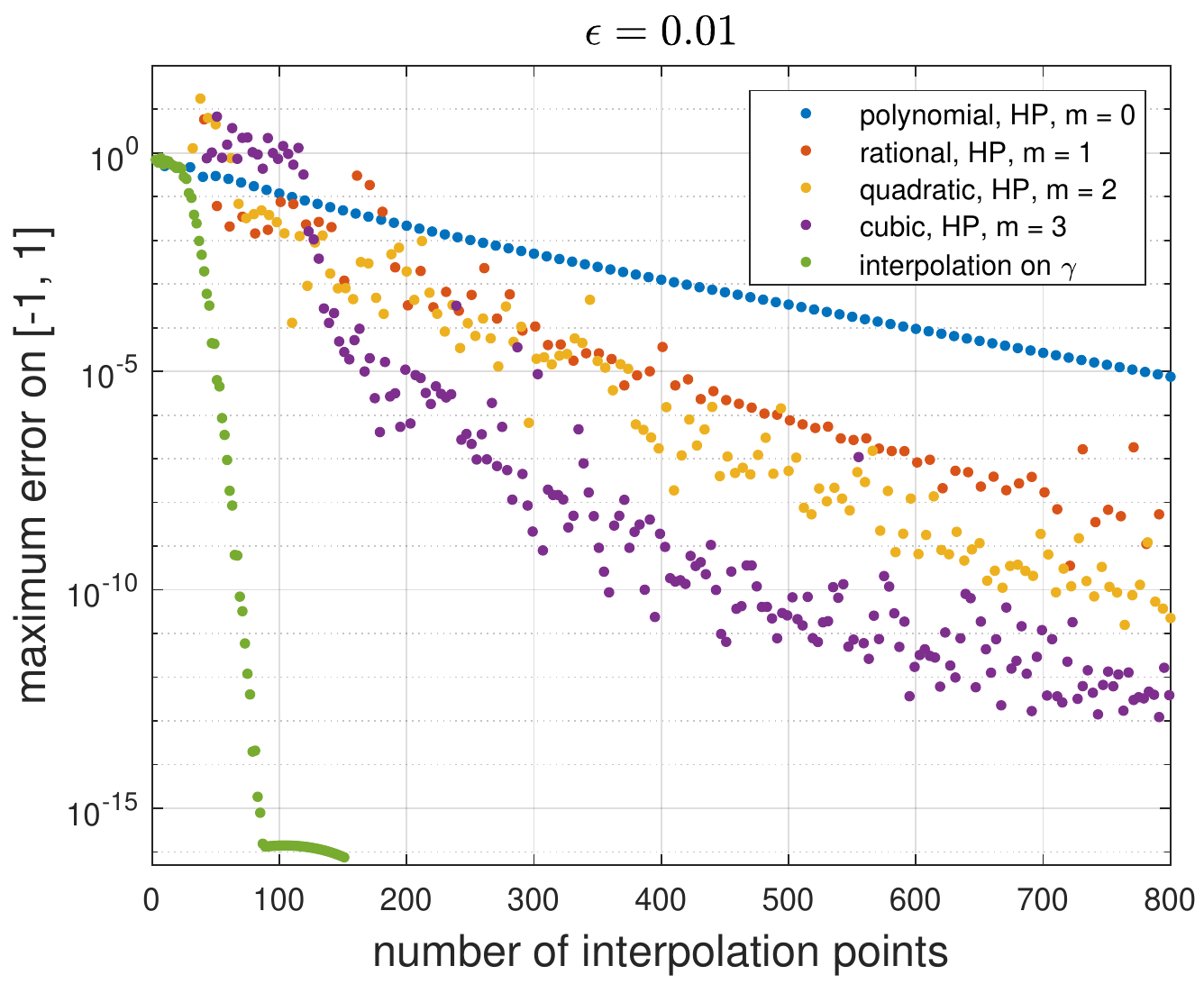}\\
    \includegraphics[width=0.495\textwidth]{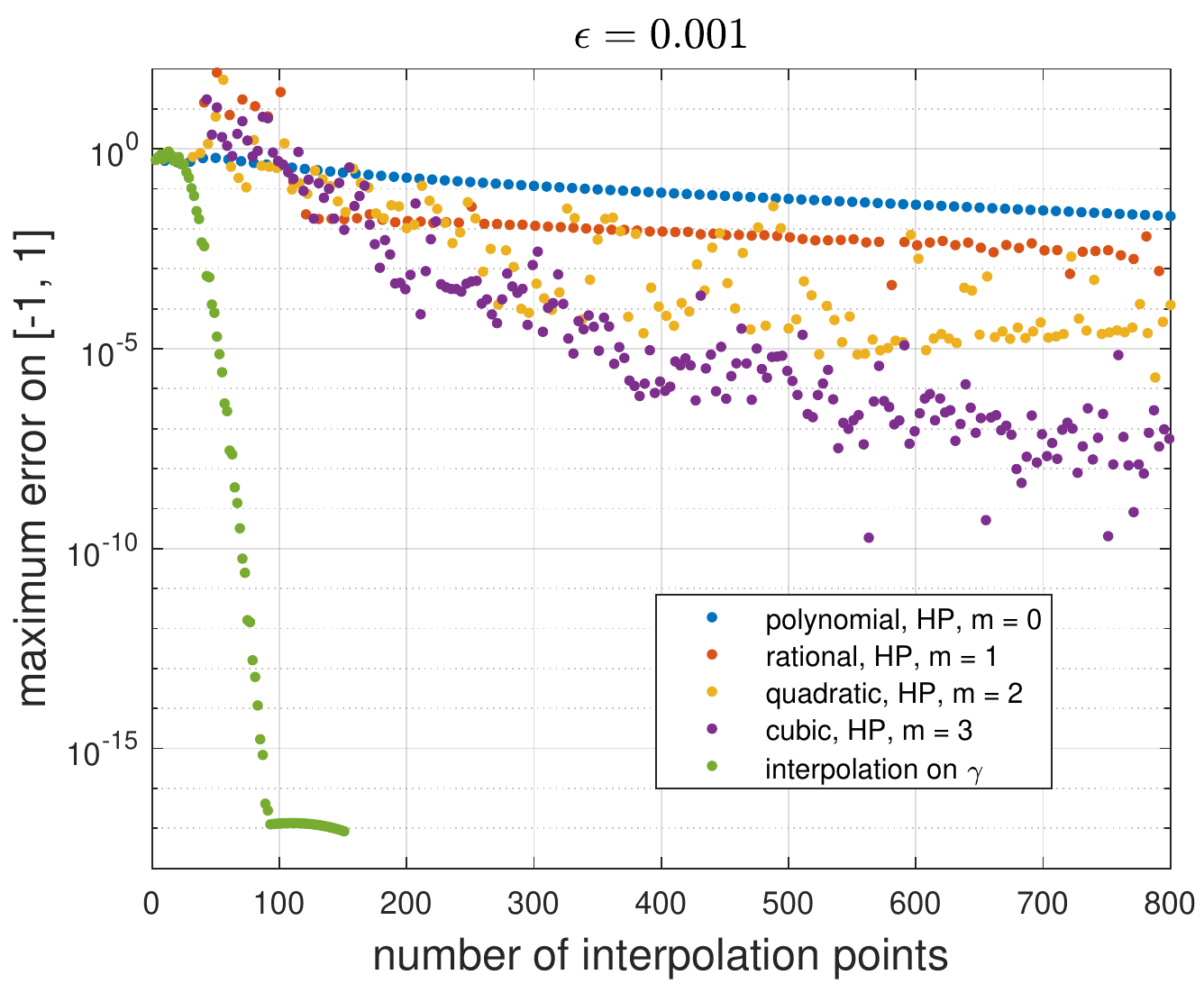}
    \includegraphics[width=0.495\textwidth]{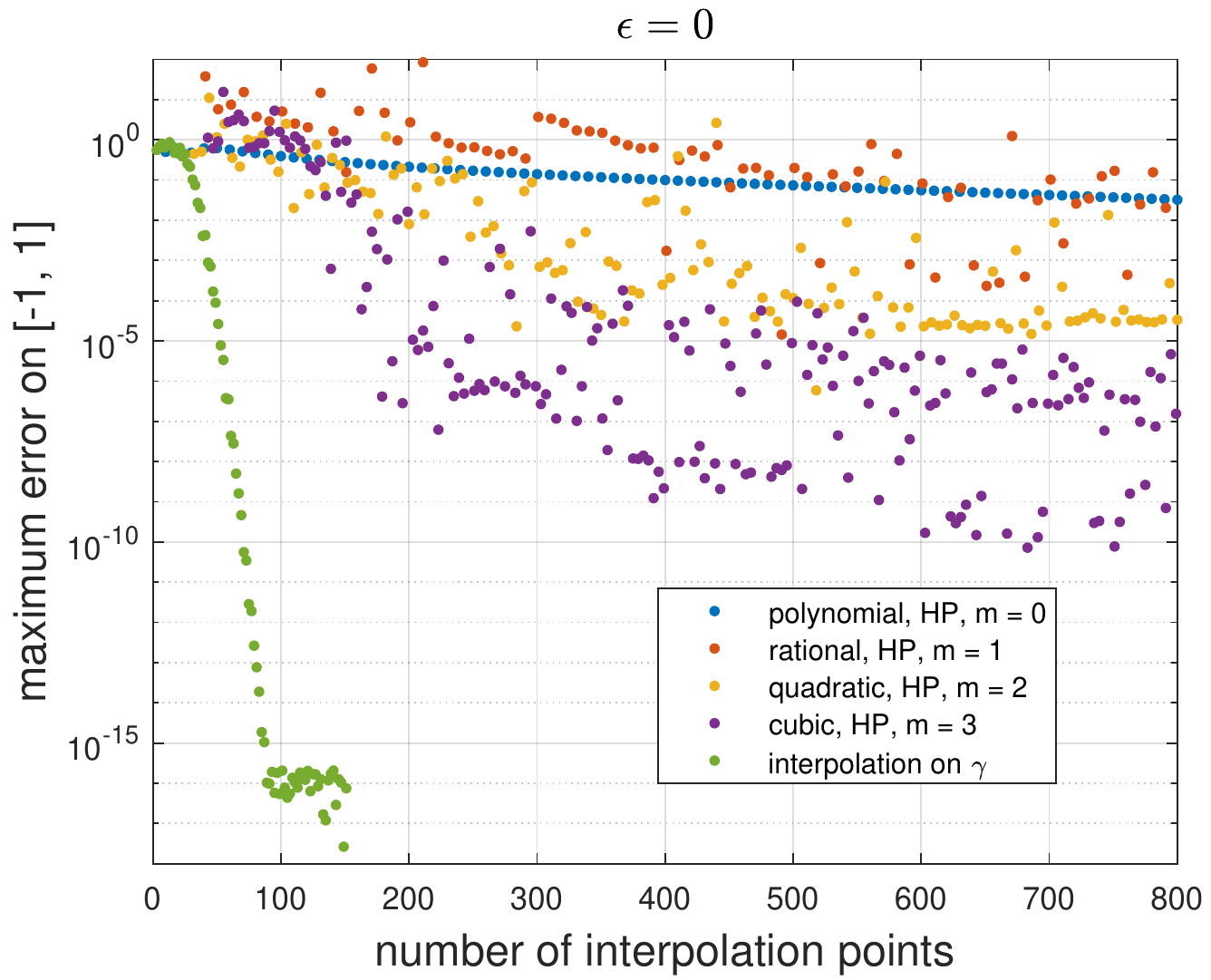}
  \end{center} 
  \caption{The function $g(t) = J_1(10t + 20\sqrt[3]{t^2 + \epsilon^2})$ (top-left) and the  rates of convergence of interpolants of $g$ for $\epsilon = 0.01, 0.001, 0$.  The interpolant based on OPs on the cubic curve converges faster and is more stable than other methods,  especially as $\epsilon \rightarrow 0$.}\label{fig:funapproxex1}
\end{figure}

\subsection{Differential equations on cubic curves}\label{sect:diffeq}

The Bessel function (\ref{eq:funapproxex1}) and other singular functions of the form $f(x,\sqrt{\phi(x)})$ or $f(\phi^{-1}(x^2),x)$ can arise as solutions to linear differential equations of the form
\begin{align}
\sum_{\lambda=0}^{m} a_{\l}(x,y)\frac{\mathrm{d}^{\l}}{\mathrm{d}x^{\l}}u(x) = g(x,y),  \label{eq:gamdesing}
\end{align}
for $y^2 = \phi(x)$, subject to boundary conditions. As discussed in section~\ref{sect:OPseg}, we can let $x \in \mathrm{supp}(w) = [-1, 1] \subset \Omega_{\g}$  or $x \in \mathrm{supp}(w)= [0, \infty)  \subset \Omega_{\g}$. For now, we consider $x \in [-1, 1]$.

Recall that $\phi$ may vanish at the endpoint(s) of $\Omega_{\g}$. If $\phi$ vanishes on $\mathrm{supp}(w)$, which can only happen at the endpoint(s) of $\mathrm{supp}(w)$, the differential equation is singular in $x$. This is to be expected since linear differential equations with singular solutions must be singular. However, this does not imply that the solution will be singular on $\g$.  For example, the Bessel function (\ref{eq:funapproxex1}) is the solution to a differential equation of the form (\ref{eq:gamdesing}) that is a singular function of $x$ but an entire function of $x$ and $y$ on the curve (\ref{eq:cubiccurveex1}).

For the orthogonal basis with respect to the inner product $[\cdot, \cdot]_{\g,w}$, which consists of only $p_n(w)$, solutions to (\ref{eq:gamdesing}) can be computed with the ultraspherical spectral method~\cite{OT13} if the $p_n(w)$ are chosen to be the Chebyshev polynomials. This approach represents (\ref{eq:gamdesing}) as banded matrices in coefficient space, in contrast to the dense matrices that arise in collocation methods. It is also possible to devise a spectral method with banded operators for the basis consisting of $p_n(w)$ and $p_n(\phi w)$. For this basis, unlike the basis consisting solely of $p_n(w)$, the operators are not known explicitly but can be constructed efficiently with quadrature.

 In the examples that follow, however, we compute solutions to (\ref{eq:gamdesing}) with a spectral collocation method
 since we found that it was just as accurate (for the first example below) or more accurate (for the second example) than the ultraspherical spectral method\footnote{The well-conditioning of the ultraspherical method is established in~\cite{OT13} for non-singular linear differential equations. Numerical evidence in~\cite{C19} illustrated superior conditioning and accuracy of the ultraspherical method compared to collocation for a linear equation with regular singularities. An analysis of the ultraspehrical and collocation method for equations with irregular singularities (which can include equations of the form (\ref{eq:gamdesing})) remains an open problem which is beyond the scope of the present paper. }. In particular, we approximate the solution as 
 \begin{align}
 u(x,y) \approx u_0^n Y_0 + u_{1,1}^n Y_{1,1} + u_{1,2}^n Y_{1,2}  + \sum_{k=2}^{n-1} \sum_{i=1}^3 u_{k,i}^n Y_{n,i} +  u_{n,2}^n Y_{n,2} + u_{n,3}^n Y_{n,3},  \label{eq:nevenexp}
 \end{align}
 if $n$ is odd (cf.~(\ref{eq:oddset})) and
\begin{align}\label{eq:noddexp}
 u(x,y) \approx \, & u_0^n Y_0 + u_{1,1}^n Y_{1,1} + u_{1,2}^n Y_{1,2} \\
     & + \sum_{k=2}^{n-1} \sum_{i=1}^3 u_{k,i}^n Y_{n,i} +  u_{n,2}^n Y_{n,2} + u_{n,3}^n Y_{n,3} + u_{n+1,3}^n Y_{n+1,3},  
\notag
\end{align}
 if $n$ is even (cf.~(\ref{eq:evenset})). It follows from (\ref{eq:oddset}) and (\ref{eq:evenset}) that an equivalent representation of the right-hand sides of (\ref{eq:nevenexp}) and (\ref{eq:noddexp}) is given by  the right-hand side of (\ref{eq:interpexp}) (for the basis consisting of only $p_n(w)$) or (\ref{eq:b1interpexp}) (for the basis constructed from $p_n(w)$ and $p_n(\phi w)$). The coefficients $u_{k,i}^n$ (or, equivalently, $a_{k,N}$ and $b_{k,N}$ in (\ref{eq:interpexp}) or (\ref{eq:b1interpexp})) are determined by solving a linear system so that the ODE is satisfied on a grid. We let the grid be $(x_{k,N},\pm y_{k,N})$, $k = 1, \ldots, N$, where the $x_{k,N}$ are the roots of $p_N(w)$ and $y_{k,N} = \sqrt{\phi(x_{k,N})}$. Recall that if $n = 2m + 1$, then $N = 3m +1$ and if $n = 2m$, then $N = 3m$.
 
To implement the collocation method, we use the fact that 
\begin{align}\label{eq:collder}
\frac{\mathrm{d}}{\mathrm{d}x}u(x,y)  &\approx \frac{\mathrm{d}}{\mathrm{d}x} \left[ \sum_{k=0}^{N-1} a_{k,N}p_k(w) + y\sum_{k=0}^{N-1} b_{k,N}p_k(w)     \right]\\
& = \sum_{k=0}^{N-1} a_{k,N}p_k'(w) + y\sum_{k=0}^{N-1} b_{k,N}\left(p_k'(w) + \frac{\phi'}{2\phi}p_k'(w) \right),  \notag 
\end{align}
and higher order derivatives can be derived recursively in a similar manner. Note that since the collocation points $x_{k,N}$ are in the interior of $\mathrm{supp}(w)$, $\phi(x_{k,N}) \neq 0$ and hence (\ref{eq:collder}) is well defined at the collocation points.

\subsubsection{Example 1: elliptic integral}
The elliptic integral
\begin{align*}
\int_{-1}^{1}\frac{\mathrm{d}x}{\sqrt{(x+1+\epsilon)(x+2)(x+3)}}, \qquad \epsilon > 0,
\end{align*}
which can be expressed in terms of $F(\alpha,k)$, Legendre's incomplete integral of the first kind~\cite[Ch.~19]{DLMF}, is the solution to the differential equation 
\begin{align}
y \frac{\mathrm{d}}{\mathrm{d}x} u(x,y) = 1, \qquad x \in [-1, 1], \qquad y^2 = \phi(x) = (x+1+\epsilon)(x+2)(x+3), \label{eq:ell_int_sol} 
\end{align}
subject to $u(-1,\sqrt{\phi(-1)}) = 0$ and evaluated at $(x,y) = (1,\sqrt{\phi(1)})$. For $\epsilon = 0$, the integral can be computed efficiently with Gauss--Jacobi quadrature because the singularity of the integrand at $x = -1$ is incorporated into the weight and the remainder of the integrand is analytic on $[-1, 1]$. If the singularity is close to $-1$, i.e., for $0 < \epsilon \ll 1$, Gauss--Jacobi quadrature is expensive because one needs a high degree polynomial to resolve the nearly singular integrand. By solving the differential equation on the curve $y^2 = \phi$, the computational cost of computing the integral is essentially independent of $\epsilon$ because the coefficients of the differential equation are polynomials (and hence analytic) on the curve for all $\epsilon \geq 0$. Figure~\ref{fig:ellipt_int} illustrates a solution to (\ref{eq:ell_int_sol}) which is accurate to 15 digits. 

\begin{figure}[h]
  \begin{center} 
    \includegraphics[width=0.495\textwidth]{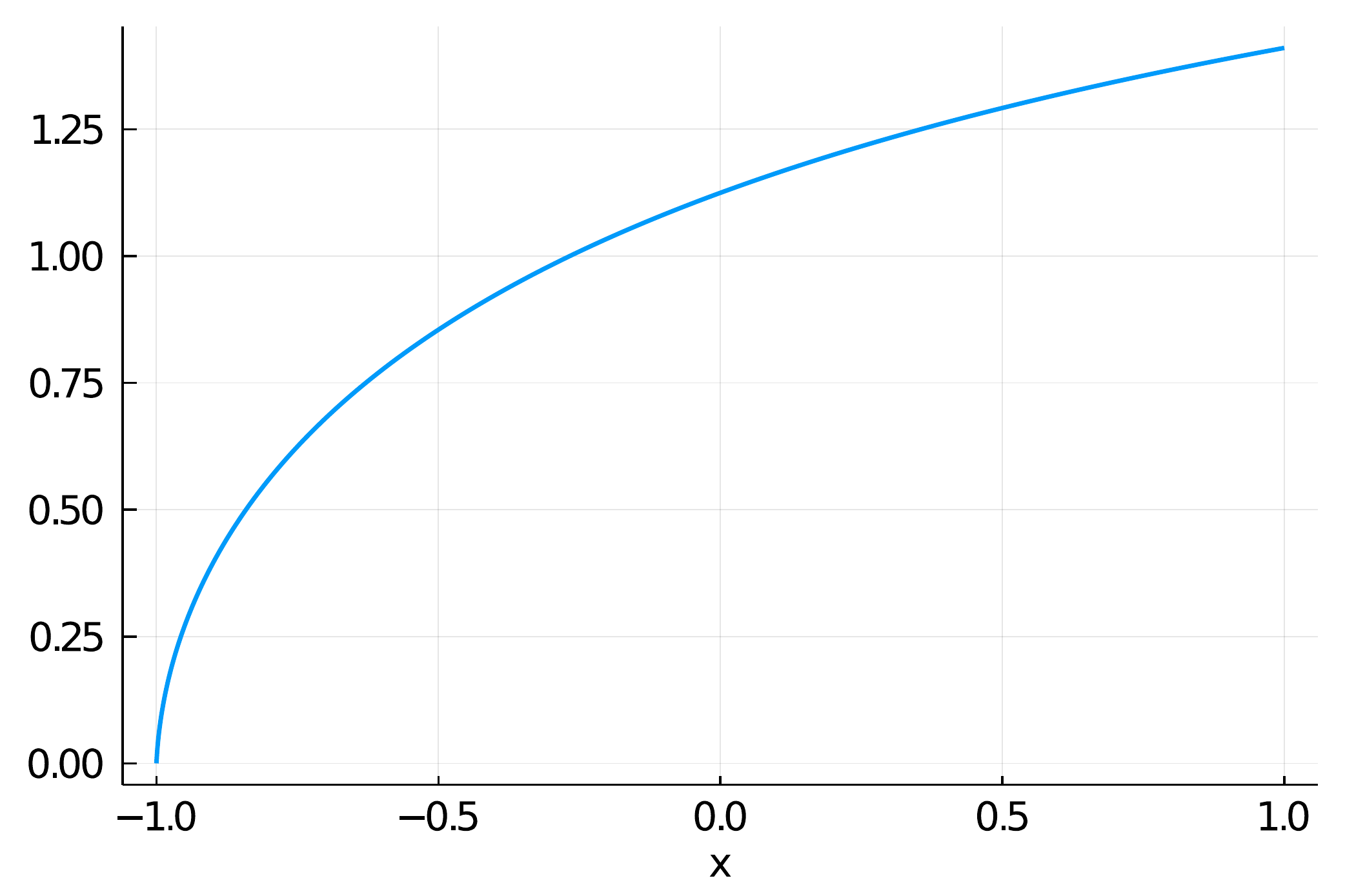}
    \includegraphics[width=0.495\textwidth]{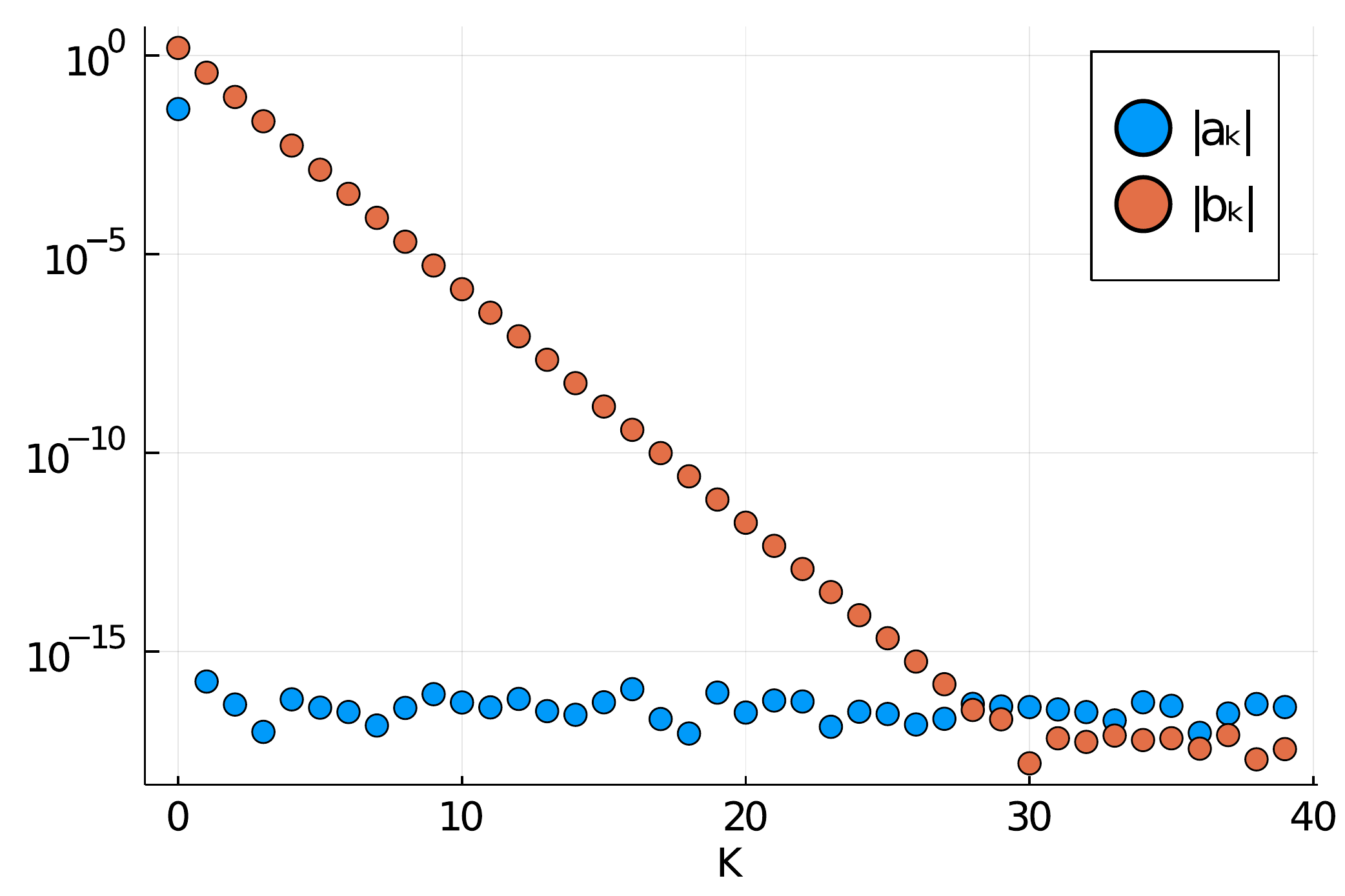}
  \end{center}
  \caption{Left:  Legendre's incomplete integral of the first kind computed by solving  \eqref{eq:ell_int_sol} with $\epsilon = 0.001$ subject to  $u(-1,\sqrt{\phi(-1)}) = 0$. Right: the magnitude of the coefficients $a_k$ and $b_k$, defined in (\ref{eq:interpexp}), of the interpolant that approximates the exact solution, obtained with a spectral collocation method. 
   }\label{fig:ellipt_int}
\end{figure}

\subsubsection{Example 2: variable coefficient, singular differential equation}
The function 
\begin{align*}
f(x) = \sin\left(c_1x + c_2x \sqrt{(1-x^2)(2-x)}\right), \qquad x \in [-1, 1],
\end{align*}
has square-root singularities at $x = \pm 1$. On the cubic curve
\begin{align*}
\g = \left\lbrace (x,y) : x\in [-1, 1],  y^2 = \phi(x) =  (1 - x^2)(2 - x)   \right\rbrace,
\end{align*}
the function becomes
\begin{align}
f(x) = u(x,y) = \sin(c_1x + c_2x y),  \label{eq:de2sol}
\end{align}
and satisfies a second order equation of the form (\ref{eq:gamdesing}) with
\begin{align}
\begin{split}
y\phi a_2(x,y) & = \frac{c_2 x}{2}\phi\phi' + c_2 \phi^2 + y c_1 \phi	 \in \Pi_6(\g), \\
y\phi a_1(x,y) & = -y\phi\left(\partial_x + \frac{\phi'}{2y} \partial_y\right) a_2(x,y) \in \Pi_5(\g), \\
y\phi a_0(x,y) & = y\phi a_2^3(x,y) \in \Pi_{9}(\g).
\end{split}\label{eq:decoeffs}
\end{align}
\begin{figure}[ht]
  \begin{center}
    \includegraphics[width=0.495\textwidth]{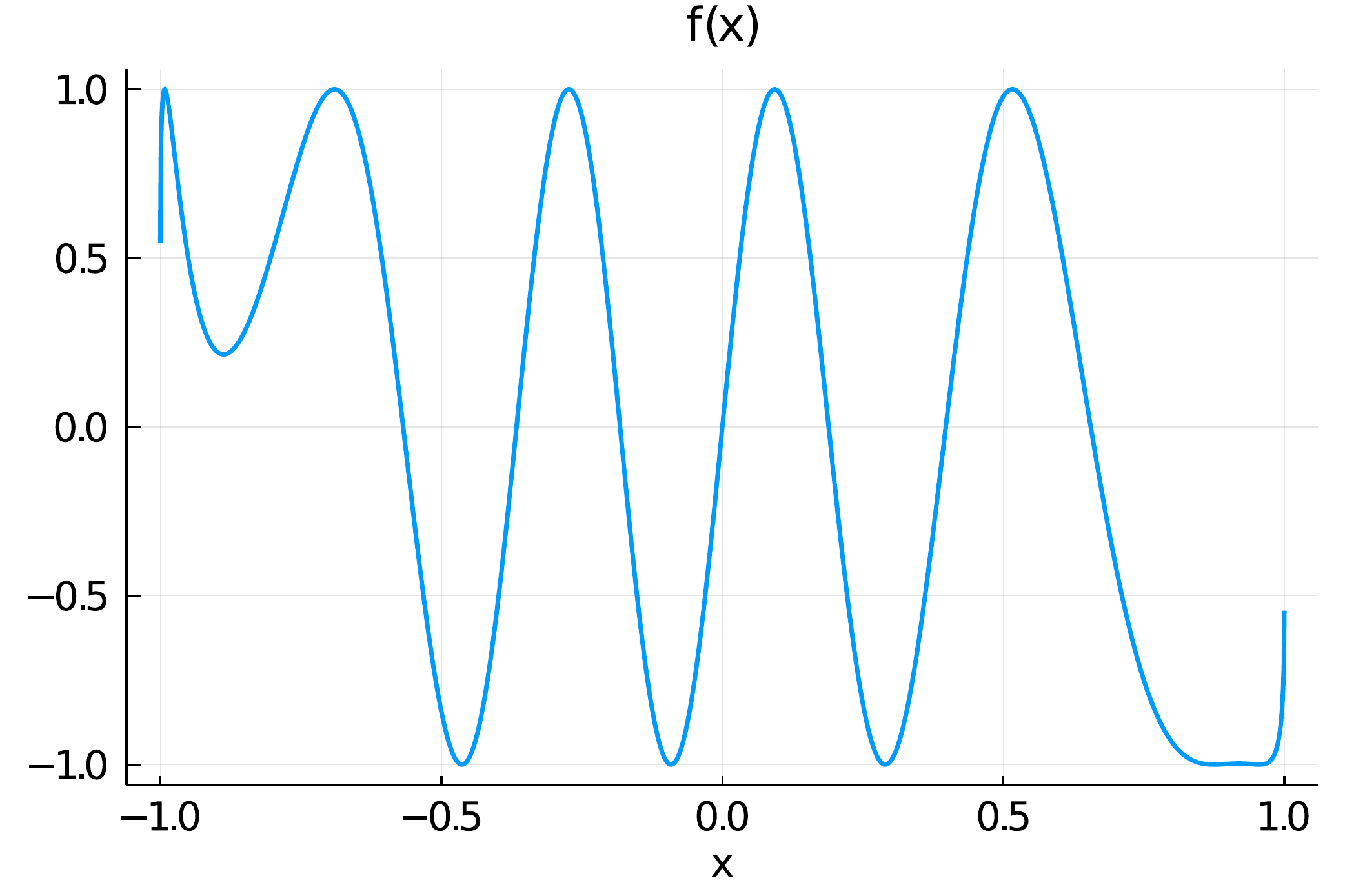}
    \includegraphics[width=0.495\textwidth]{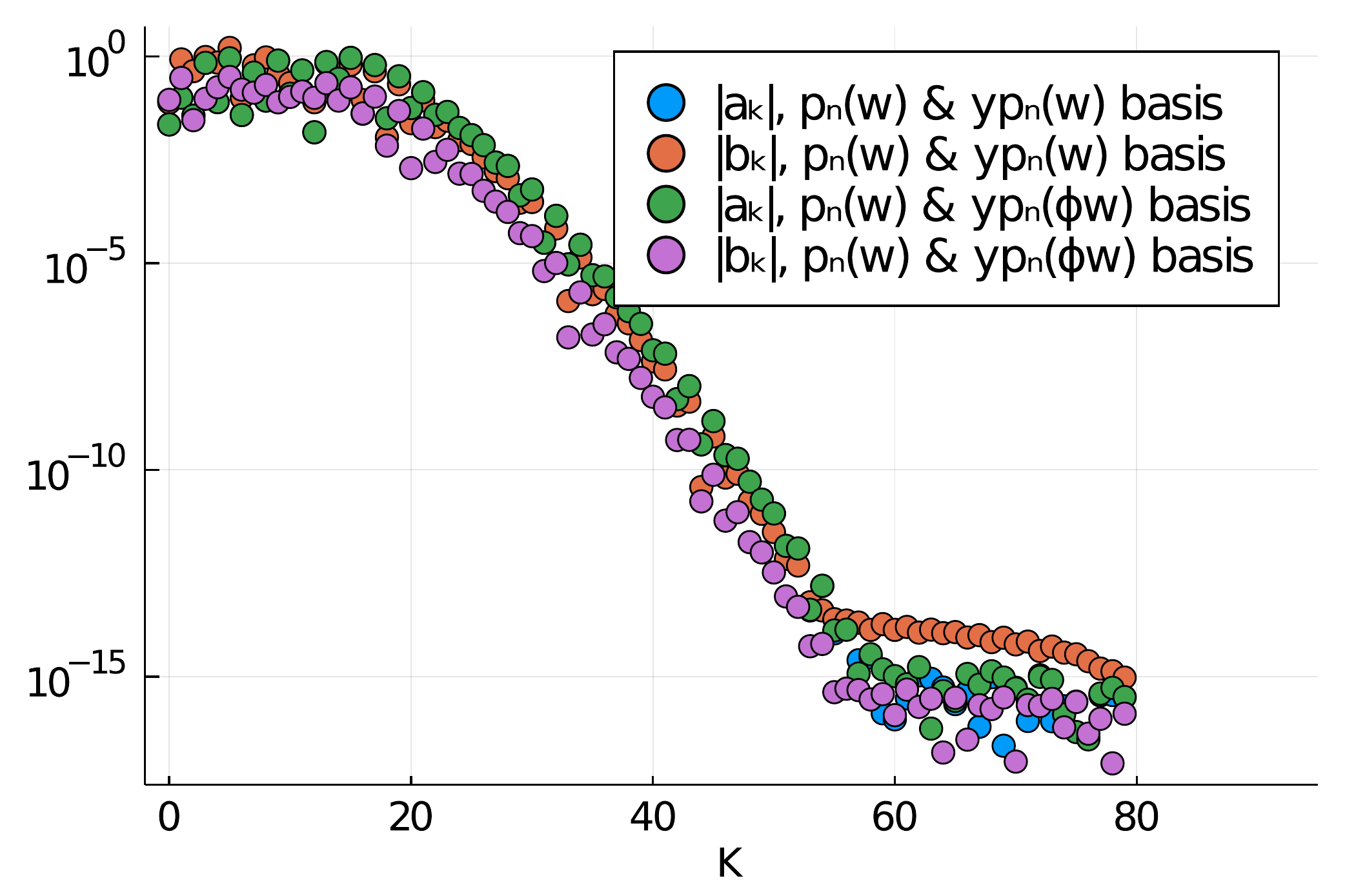}\\
    \includegraphics[width=0.6\textwidth]{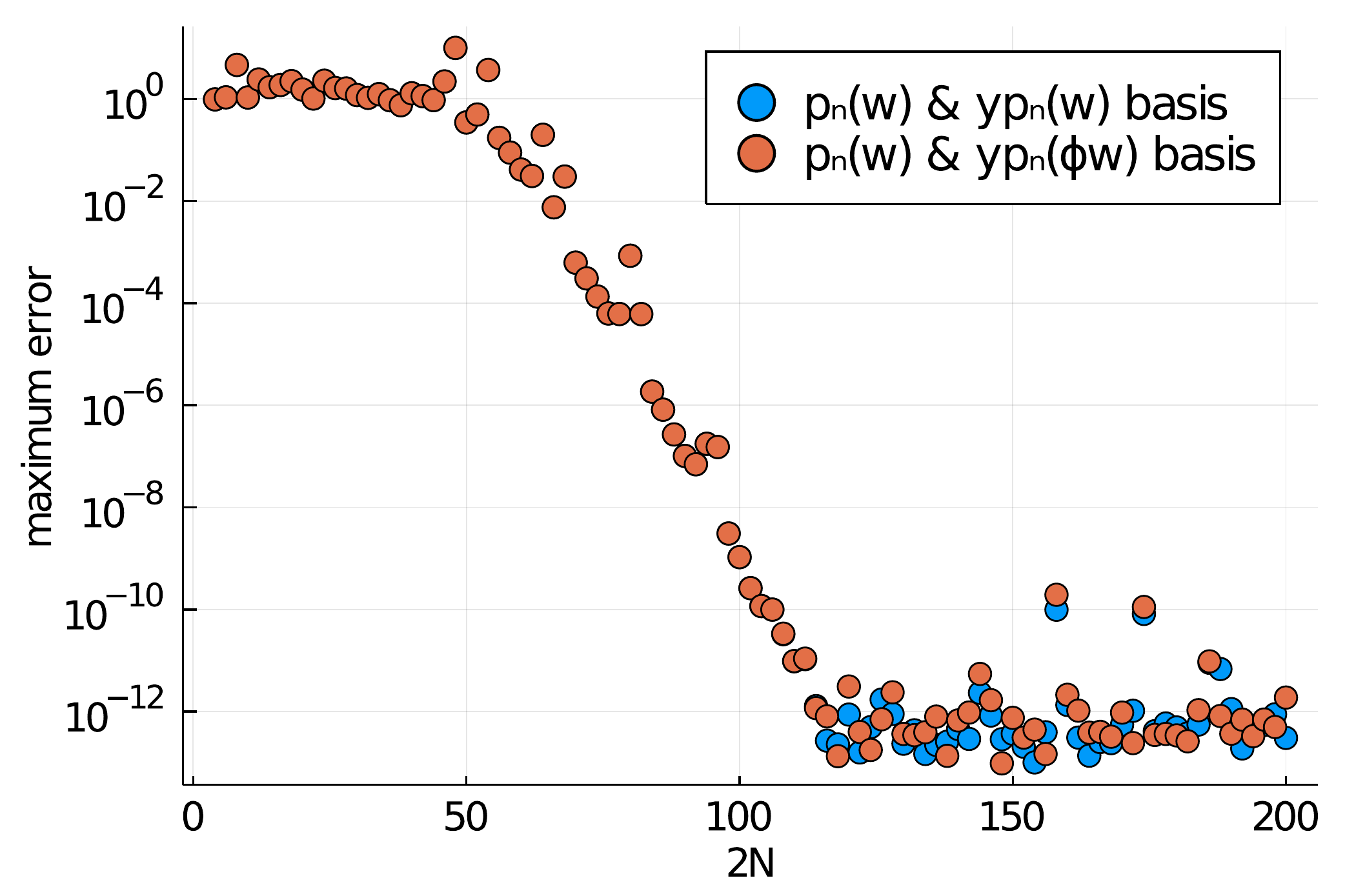}
  \end{center}
  \caption{Top-left: the function $\sin\left(c_1x + c_2x \sqrt{(1-x^2)(2-x)}\right)$ with $c_1 = 10$ and $c_2 = 5$, which is the solution to a second order differential equation of the form (\ref{eq:gamdesing}) with coefficients given in (\ref{eq:decoeffs}). Top-right: the coefficients (defined in (\ref{eq:interpexp}) and (\ref{eq:b1interpexp})) obtained by solving the differential equation with a collocation method in the basis consisting of (i) $p_n(w)$ and $yp_n(w)$ and (ii) $p_n(w)$ and $yp_n(\phi w)$ with $w = 1$. Bottom: the maximum error of the collocation solutions on $[-1, 1]$ with $2N$ coefficients.
  }\label{fig:singdesol}
  \end{figure}
Figure~\ref{fig:singdesol} shows $f(x)$ and the accuracy of the approximate solutions  obtained by solving the differential equation, subject to Dirichlet boundary conditions, with a collocation method using two orthogonal bases (one in the inner product space defined by $\la \cdot, \cdot \ra_{\g,w}$ and the other in the space defined by $[\cdot, \cdot]_{\g,w}$). The accuracy obtained with both bases are comparable and roughly four digits of accuracy are lost, most likely due to the large condition numbers of the matrices that arise in the collocation methods.
The coefficients decay super-exponentially because the solution is an entire function of $x$ and $y$ on the $\g$. By contrast, a conventional spectral method using an orthogonal basis on $[-1, 1]$ could only achieve algebraic convergence to the solution because of its singularities at $x = \pm 1$.

\section{Conclusion}

We have constructed orthogonal polynomials on cubic curves and demonstrated that they can be used for function approximation and solving differential equations. There are some clear questions left to explore:

\begin{enumerate}
\item Higher order algebraic curves. In the case where there is symmetry upon reflection across the $x$ axis we may be able to reduce to one-dimensional orthogonal polynomials, but a general construction is as-of-yet unclear.
\item Applications to partial differential equations. We mention that the construction of OPs on quadratic curves in \cite{OX2,OX3} led to numerical methods for partial differential equations on non-standard disk-slices and trapeziums in \cite{SO} and explicit formul\ae\ for inverting wave operators on a cone \cite{OX4}. This suggests that the OPs on cubic curves  introduced here may be applicable to solving partial differential equations whose boundaries are defined in terms of a cubic curves.
\end{enumerate}

\appendix


\begin{thebibliography}{99}

\bibitem{bix} 
R. Bix,
\textit{Conics and Cubics:
A Concrete Introduction to
Algebraic Curves}, second edition, Springer, 2006.

\bibitem{BM}
A. Bonifant and J. Milnor, 
 On real and complex cubic curves. 
\textit{Enseign. Math.} \textbf{63} (2017), 21--61.


\bibitem{C19}
S. Crespo, M. Fasondini, C. Klein, N. Stoilov, and C. Vall\'ee,
  Multidomain spectral method for the Gauss hypergeometric function,
  \textit{Numer. Algorithms},
  \textbf{84} (2020), 1--35.

\bibitem{DLMF} 
        \textit{NIST Digital Library of Mathematical Functions}.
         http://dlmf.nist.gov/, 
         Release 1.0.27 of 2020-06-15. F. W. J. Olver, A. B. Olde Daalhuis, D. W. Lozier, B. I. Schneider,
         R. F. Boisvert, C. W. Clark, B. R. Miller, B. V. Saunders, H. S. Cohl, and M. A. McClain, eds.

\bibitem{DX} 
        C. F. Dunkl and Y. Xu,
        \textit{Orthogonal Polynomials of Several Variables}, 2nd ed. 
        Encyclopedia of Mathematics and its Applications \textbf{155},
        Cambridge University Press, Cambridge, 2014.
         
\bibitem{QP} 
        M. Fasondini, N. Hale, R. Spoerer, and J.A.C. Weideman,
        Quadratic {P}ad{\'e} approximation: numerical aspects and applications,
        \textit{Comput. Res. Mod.}, \textbf{11} (6) (2019), 1017--1031.
         
\bibitem{gautschi}
        W. Gautschi,
        \textit{Orthogonal Polynomials: Computation and Approximation},
        Oxford University Press, 2004.
  
\bibitem{gonnet}
  P. Gonnet, R. Pach\'on, and L.N. Trefethen,
  Robust rational interpolation and least-squares,
  \textit{Elect. Trans. Numer. Anal.},
  \textbf{38} (2011), 146--167.
  
\bibitem{HT}
N. Hale and A. Townsend,
  A fast, simple, and stable Chebyshev--Legendre transform using an asymptotic formula,
  \textit{SIAM J. Sci. Comput.},
  \textbf{36} (1) 2014, A148--A167.
  
\bibitem{HT16}
N. Hale and A. Townsend,
A fast FFT-based discrete Legendre transform,
  \textit{IMA J. Numer. Anal.},
 \textbf{36}(4) (2016), {1670--1684}.
   
 \bibitem{Ko}
         N. Koblitz, 
         Introduction to elliptic curves and modular forms, 
         Springer, 1993. 
         
\bibitem{OT13}
S. Olver and A. Townsend,
A fast and well-conditioned spectral method, 
\textit{SIAM Review},
\textbf{55}(3) 2013, 462--489.
          
          
\bibitem{Ol} 
          S. Olver,  
          https://approximatelyfunctioning.blogspot.com/2020/09/quasi-matrices-orthogonal-polynomials.html 
          
          \bibitem{O4}
          S. Olver, OrthogonalPolynomialsQuasi.jl v0.4.0. Available at
          https://github.com/JuliaApproximation/OrthogonalPolynomialsQuasi.jl    

\bibitem{OX1}
         S. Olver and Y. Xu,
         Orthogonal structure on a wedge and on the boundary of a square, 
         \textit{Found. Comp. Math.}, 19 (2019), 561--589.

\bibitem{OX2}
        S. Olver and Y. Xu,
        Orthogonal structure on a quadratic curve, 
        \textit{IMA J. Numer. Anal.}, to appear.
 	
\bibitem{OX3}
        S. Olver and Y. Xu,
        Orthogonal polynomials in and on a quadratic surface of revolution,
	\textit{Maths Comp.}, 89 (2020) 2847--2865.	
       
\bibitem{OX4}
       S. Olver and Y. Xu,
	Non-homogeneous wave equation on a cone, 
	\textit{Int. Trans. Spec. Funcs.}, to appear.

          
\bibitem{pachon}
      R. Pach\'on, P. Gonnet, and J. Van Deun,
      Fast and stable rational interpolation in roots of unity and Chebyshev points,
      \textit{SIAM J. Numer. Anal.}, \textbf{50} (2012), 1713--1734.


\bibitem{SO}
      B. Snowball and S. Olver, 
      Sparse spectral and p-finite element methods for partial differential equations on disk slices and trapeziums, 
      \textit{Stud. Appl. Maths}, \textbf{145}  (2020) 3--35.

     
\bibitem{X20a}
       Y. Xu, 
       Fourier series in orthogonal polynomials on a cone of revolution,
       \textit{J. Fourier Anal. Appl.}, \textbf{26} (2020), Article number: 36. 
  
\bibitem{X20b}
       Y. Xu, 
       Orthogonal structure and orthogonal series in and on a double cone or a hyperboloid.
        \textit{Trans. Amer. Math. Soc.}, in print. arXiv:1912.07533


      
\end{thebibliography}
\end{document}